\documentclass[12pt,reqno]{amsart}

\usepackage{amsfonts,color,amsthm,amsmath,amssymb}

\usepackage{color}
\allowdisplaybreaks[3]

\def\Box{\vcenter{\vbox{\hrule\hbox{\vrule
     \vbox to 8.8pt{\hbox to 10pt{}\vfill}\vrule}\hrule}}}

\newcommand{\tr}{\textup{Tr}}

\newcommand{\ra}{\rangle}
\newcommand{\la}{\langle}
\newcommand{\F}{{\mathbb F}}

\newcommand{\cC}{{\mathcal C}}
\newcommand{\cH}{{\mathcal H}}
\newcommand{\cS}{{\mathcal S}}
\newcommand{\cN}{{\mathcal N}}
\newcommand{\Aut}{\textup{Aut}}
\newcommand{\rank}{\textup{rank}}
\newcommand{\rad}{\textup{rad}}

\newcommand{\PG}{\textup{PG}}

\newtheorem{thm}{Theorem}

\newtheorem{lemma}[thm]{Lemma}
\newtheorem{corollary}[thm]{Corollary}

\newtheorem{example}[thm]{Example}
\numberwithin{equation}{section}

\newtheorem{remark}[thm]{Remark}

\numberwithin{thm}{section}

\begin{document}
\title[flag-transitive affine planes]{Finite flag-transitive affine planes with a solvable automorphism group}

\author{Tao Feng}
\address{School of Mathematical Sciences, Zhejiang University, Hangzhou 310027, Zhejiang, China}
\email{tfeng@zju.edu.cn}

\keywords{spread, affine plane, planar function, commutative semifield, quadratic form, permutation polynomial.}
\begin{abstract}
In this paper, we consider finite flag-transitive affine planes with a solvable automorphism group. Under a mild number-theoretic condition involving the order and dimension of the plane, the translation complement must contain a linear cyclic subgroup that either is transitive or has two equal-sized orbits on the line at infinity. We develop a new approach to the study of such planes by associating them with planar functions and permutation polynomials in the odd order and even order case respectively. In the odd order case, we characterize the Kantor-Suetake family by using Menichetti's classification of generalized twisted fields and Blokhuis, Lavrauw and Ball's classifcation of rank two commutative semifields. In the even order case, we develop  a technique to study permutation polynomials of DO type by quadratic forms and characterize such planes that have dimensions up to four over their kernels.
\end{abstract}

\maketitle

\section{Introduction}

Let $V$ be a $2n$-dimensional vector space over the finite field $\F_q$. A {\it spread} $\cS$ of $V$ is a collection  of $n$-dimensional subspaces that partitions the nonzero vectors in $V$. The members of $\cS$ are the {\it components}, and $V$ is the {\it ambient space}. The {\it kernel} is the subring of $\Gamma L(V)$ that fixes each component, and it is a finite field containing $\F_q$. The {\it dimension} of $\cS$ is the common value of the dimensions of its components over the kernel. The {\it automorphism group} $\Aut(\cS)$  is the subgroup of $\Gamma L(V)$ that maps components to components. The incidence structure $\Pi(\cS)$ with point set $V$ and line set $\{W+v:\,W\in\cS,\,v\in V\}$ and incidence being inclusion is a translation plane. The kernel or dimension of $\Pi(\cS)$ is that of  $\cS$ respectively. Andre \cite{andre_tr} has shown that $\Aut(\cS)$ is the translation complement of the plane $\Pi(\cS)$  and each finite translation plane can be obtained from a spread in this way. Two spreads $\cS$ and $\cS'$ of $V$ are {\it isomorphic} if $\cS'=\{g(W):\,W\in\cS\}$ for some $g\in\Gamma L(V)$, and isomorphic spreads correspond to isomorphic planes.

An affine plane is called {\it flag-transitive} if it admits a collineation group which acts transitively on the flags, namely, the incident point-line pairs. Throughout this paper, we will only consider finite planes. Wagner \cite{wagner} has shown that finite flag-transitive planes are necessarily translation planes, so the plane must have prime power order and can be constructed from a spread $\cS$ with ambient space $V$ of dimension $2n$ over $\F_q$ for some $n$ and $q$. The affine plane $\Pi(\cS)$ constructed from a spread $\cS$ is flag-transitive if and only if $\Aut(\cS)$ is transitive on the components.  Foulser has determined all flag-transitive groups of finite affine planes in \cite{foulser,foulser0}. The only non-Desarguesian flag-transitive affine planes with nonsolvable collineation groups are the nearfield planes of order $9$, the Hering plane of order $27$ \cite{hering}, and the L$\ddot{u}$neburg planes of even order \cite{luneburg}, cf. \cite{ls,kantor_h}. In the solvable case, Foulser has shown that with a finite number of exceptions, which are explicitly described, a solvable flag transitive group of a finite affine plane is a subgroup of a one-dimensional Desarguesian affine plane.

Kantor and Suetake have constructed non-Desuarguesian flag-transitive affine planes of odd order in \cite{kantor_odd,ks,suetake, suetake2}, and we will refer to these planes as  the Kantor-Suetake family. The dimension two case is also due to Baker and Ebert \cite{be_const}. Kantor and Williams have constructed large numbers of flag-transitive affine planes of even order  arising from symplectic spreads in \cite{kantor_even,kwnew}. The dimensions of these planes over their kernels are odd. It remains open whether there is a non-Desarguesian flag-transitive affine plane of even order whose dimension over its kernel is even and greater than $2$. Prince has completed the determination of all the flag-transitive affine planes of order at most $125$ in \cite{prince}, and there are only the known ones.

Except for the L$\ddot{u}$neburg planes  and  the Hering plane of order $27$, all the known finite non-Desarguesian  flag-transitive affine planes have a translation complement which contains a linear cyclic subgroup that either is transitive or has two equal-sized orbits on the line at infinity.  Under a mild number-theoretic condition involving the order and dimension of the plane (see Lemma \ref{lemma_gcd} below), it can be shown that one of these actions must occur. We call flag-transitive planes of the first kind {\it $\mathcal{C}$-planes} and those of the second kind {\it $\mathcal{H}$-planes}, and call the corresponding spreads of {\it type $\mathcal{C}$} and {\it type $\mathcal{H}$} respectively.  There has been extensive study on these two types of planes in the literature. In the case the plane has odd order and dimension two or three over its kernel, it has been shown that the known examples are the only possibilities for either of these two types, see \cite{be_last,be_baer2,be_baer,be_2dim,be_x,ebert_sur}. The classification takes a geometric approach by making use of the intersection of the corresponding spread with the orbits of a certain Singer subgroup and considering relevant Baer subgeometry partitions.

In the study of finite flag-transitive projective planes, deep results from finite group theory are invoked and considerable progress has been made towards a complete classification, cf. \cite{kantor_proj,Kthas_proj, tz_proj}. In contrast, the affine case is more of a combinatorial flavor, and a complete classification seems far out of reach. In this paper, we show that there is a broader connection between flag-transitive affine planes and other combinatorial objects than that is previously known. This will lead us to new characterization results on such planes by making use of the deep results already obtained in other circumstances. To be specific, in Section 3 we develop a new approach to the study of such planes by associating them with planar functions and permutation polynomials in the odd order and even order case respectively. In the odd order case, this new approach allows us to characterize the Kantor-Suetake family by making use of Menichetti's classification of generalized twisted fields in \cite{meni1,meni2}. In particular, the cases of dimension two and three over their kernels follow as a consequence. In Section 4 we will consider the nuclei of the associated commutative semifields and study planar functions that correspond to rank two commutative semifields by the classification results of such semifields by Blokhuis, Lavrauw and Ball in \cite{rtcs_char1,rtcs_char2}. In the even order case, we will develop  a technique to study permutation polynomials of DO type by quadratic forms and characterize such planes that have dimensions up to four over their kernels in Section 5. This is the first characterization result in the even order case to our knowledge.

\section{Preliminaries}

A finite {\it presemifield} $S$ is a finite ring with no zero-divisors such that both the left and right distributive laws hold. If further it contains a multiplicative identity, then we call $S$ a {\it semifield}. The additive group of a finite presemifield is necessarily an elementary abelian $p$-group for some prime $p$, so it is conventional to identify $(S,+)$ as the additive group of the finite field $\F_q$ with $q=|S|$ elements. Two presemifields $S_1=(\F_q,+,\star)$ and $S_2=(\F_q,+,\ast)$ are {\it isotopic} if there exist three linear bijections $L,\,M,\,N$ from $(\F_q,+)$ to itself such that $M(x)\ast N(y)=L(x\star y)$ for all $x,y\in \F_q$. In this case, we call $S_1$ an {\it isotope} of $S_2$ and vice versa. For a presemifield $S=(\F_q,+,*)$, fix any $e\ne 0$ and define a new multiplication $\circ$ by $(x*e)\circ(e*y)=x*y$. Then $S'=(\F_q,+,\circ)$ is a semifield with multiplicative identity $e*e$ that is isotopic to $S$.

Let $S=(\F_q,+,*)$ be a commutative semifield. Its  {\it middle nucleus}   $\cN_m(S)$ and {\it nucleus} $\cN(S)$ are defined respectively as follows:
\begin{align*}
\cN_m(S)=\{\alpha\in \F_q:\,(x*\alpha)*y=x*(\alpha*y) \textup{ for all } x,y\},\\
\cN(S)=\{\alpha\in \F_q:\,(\alpha*x)*y=\alpha*(x*y) \textup{ for all } x,y\}.
\end{align*}
Both $\cN_m(S)$ and $\cN(S)$ are finite fields, and we can regard $S$ as a vector space over $\cN_m(S)$. For more details, please refer to \cite{handbook} or the surveys in \cite[Chapter 9]{handbookff} and \cite{fs_cur}.

A {\it rank two commutative semifield} (RTCS for short) is a commutative semifield that is of rank at most two over its middle nucleus.  A finite field is a RTCS by definition, and other known examples include: Dickson semifields \cite{dickson_r2}, Cohen-Ganley semifields \cite{CG} and the Penttila-Williams semifield \cite{rtcs_pw}. They  have close connections with many central objects in finite geometry, cf. \cite{rtcs_sur,fs_cur}. The following approach to RTCS is due to Cohen and Ganley \cite{CG}.
\begin{thm}\label{thm_rtcs_def}
Let $q$ be odd, and fix $t\in \F_{q^2}\setminus\F_q$. Let $f,\,g:\,\F_q\mapsto\F_q$ be two functions.  The algebraic system $S(g,f):=(\F_{q^2},+,\circ)$ with multiplication
\begin{equation}\label{eqn_rtcs_def}
(xt+y)\circ (ut+v)=(xv+yu+g(xu))t+yv+f(xu)
\end{equation}
is a RTCS if and only if $f,\,g$ are linear and $g(x)^2+4xf(x)$ is a nonsquare for  $x\ne 0$.
\end{thm}

In this manner, the Dickson semifield can be described as $S(g,f)$ with $g(z)=0$ and $f(z)=mz^\sigma$, where $q$ is odd, $\sigma$ is an automorphism of $\F_q$, and $m$ is a nonsquare. It is clear that a different choice of $t$ in the theorem lead to an isotopic semifield. In the following theorem we collect some characterization results of Dickson semifields among all RTCSs.
\begin{thm}[\cite{rtcs_char1,rtcs_char2}]\label{thm_rtcs}
Let $S$ be a RTCS of order $p^{2n}$, $p$ an odd prime. If the nucleus $\cN(S)=\F_q$, and $p^n=q^s$, then $S$ is isotopic to either a finite field or a Dickson semifield if  $p>2n^2-(4-2\sqrt{3})n+(3-2\sqrt{3})$, $q\ge 4s^2-8s+2$ or $s=3$.
\end{thm}

A polynomial $f(X)\in\F_{q}[X]$ is {\it reduced} if $\deg(f)\le q-1$. It is well known that any function $f:\,\F_q\mapsto\F_q$ is uniquely representable by a reduced polynomial in $\F_q[X]$, i.e., there exists a unique reduced polynomial $g(X)\in\F_q[X]$ such that $f(x)=g(x)$ for all $x\in\F_q$. We will  write $f(X)$ and $f(x)$  when we regard $f$ as a polynomial and a function respectively. A polynomial of the form $L(X)=\sum_{i=0}^{m}a_iX^{q^i}$ with coefficients $a_i$'s in $\F_{q^n}$ is called a {\it $q$-polynomial} over $\F_{q^n}$. If $q$ is not specified in the context, then it is also called a {\it linearized polynomial}. If $L(X)$ is a reduced linearized polynomial over $\F_{q^n}$, then the map $x\mapsto L(x)$ is $\F_q$-linear if and only $L(X)$ is a $q$-polynomial. If  a subspace $U$ is in the kernel of $L$, then $L_U(X)=\prod_{a\in U}(X-a)$ is a linearized polynomial, and there is a linearized polynomial $R$ such that $L(X)=R(L_U(X))$, cf. \cite[Theorem 3.52, Theorem 3.62]{ff}. For a $q$-polynomial $L(X)=\sum_{i=0}^{n-1}d_iX^{q^i}$, define its {\it adjoint} polynomial as $\tilde{L}(X)=\sum_{i=0}^{n-1}d_i^{q^{n-i}}X^{q^{n-i}}$, and  its associated matrix as
\begin{equation}\label{eqn_mat}
M:=\begin{pmatrix}
d_0& d_1  & \cdots &d_{n-1} \\
d_{n-1}^q&d_0^q  &\cdots &d_{n-2}^q \\
\vdots&\vdots&\ddots& \vdots\\
d_{1}^{q^{n-1}} &d_2^{q^{n-1}}   &\cdots&d_0^{q^{n-1}}
\end{pmatrix}.
\end{equation}
The following lemma is well-known, and we sketch a proof.
\begin{lemma}\label{lemma_M_rank} If $L(X)=\sum_{i=0}^{n-1}d_iX^{q^i}\in\F_{q^n}[X]$, then the kernel of the $\F_q$-linear map $x\mapsto L(x)$ from $\F_{q^n}$ to itself has dimension  $n-\rank(M)$, where $M$ is as defined in \eqref{eqn_mat}.
\end{lemma}
\begin{proof}Take $\{t_0,t_1,\cdots,t_{n-1}\}$ to be a basis of $\F_{q^n}$ over $\F_q$, and let $T$ be the $n\times n$ matrix whose $(i,j)$-th entry is $t_j^{q^i}$ for $0\le i,j\le n-1$. By \cite[Lemma 3.51]{ff}, $T$ is invertible. It is routine to check that
$(T^tMT)_{i,j}=\tr_{\F_{q^n}/\F_q}(t_iL(t_j))$ and $(MT{\bf x}^t)_i=L(x)^{q^i}$,
where ${\bf x}=(x_0,x_1,\cdots,x_{n-1})\in\F_q^n$ and $x=x_0t_0+x_1t_1+\cdots +x_{n-1}t_{n-1}\in \F_{q^n}$. It follows that $L(x)=0$ if and only if ${\bf x}^t$ is in the null space of $T^tMT$. The matrix $T^tMT$ has all its entries in $\F_q$, so its rank over $\F_q$ and $\F_{q^n}$ are the same. The conclusion now follows.
\end{proof}

A function $f:\,\F_{q}\mapsto\F_q$ is a {\it planar function} if $x\mapsto f(x+a)-f(x)-f(a)$ is a permutation of $\F_q$ for any $a\ne 0$. It is known that there are no planar functions in even characteristic. A Dembowski-Ostrom (or DO) polynomial over a field of characteristic $p$ is a polynomial of the shape $\sum_{i,j}a_{ij}X^{p^i+p^j}$. If a function $f:\,\F_q\mapsto\F_q$ is representable by a DO polynomial, we call $f$ a function of DO type. Two planar functions $f,g$ of DO type are {\it equivalent} if there exist linearized polynomials $L_1,\,L_2$ such that $f(x)=L_1(g(L_2(x)))$ for all $x\in\F_q$. It is shown in \cite{coulter} that there is a close connection between commutative presemifield and planar functions of DO type: if $(\F_q,+,*)$ is a commutative presemifield, then $x\mapsto x*x$ is a planar function of DO type over $\F_q$; conversely, if $f$ is a planar function of DO type, then $S_f=(\F_q,+,*)$ is a commutative presemifield, where the multiplication is  $x*y=f(x+y)-f(x)-f(y)$. Weng and Hu have given a characterization of planar functions in terms of the images of $f$ in \cite[Theorem 2.3]{weng}.
\begin{lemma}\label{lemma_2to1}
Let $f:\,\F_q\mapsto\F_q$ be a DO polynomial. Then $f$ is a planar function if and only if $f$ is $2$-to-$1$, namely, every nonzero element has $0$ or $2$ preimages.
\end{lemma}

We shall need the following simple lemma.
\begin{lemma}\label{lemma_L_perm}
Let $q$ be an odd prime power and $w$ be a nonsquare of $\F_q$. If $L_1(X)$ and $L_2(X)$ are linearized polynomials such that $Q(x)=L_1(x)^2-wL_2(x)^2$ is a planar function over $\F_q$, then at least one of $L_1$ and $L_2$ is a permutation polynomial.
\end{lemma}
\begin{proof}
If neither  $L_1$ nor $L_2$ is a permutation, then there exists nonzero elements $u_1,u_2$ such that $L_1(u_1)=L_2(u_2)=0$. However, this leads to that $Q(u_1+u_2)-Q(u_1)-Q(u_2)=0$, which contradicts the planarity property.
\end{proof}

A {\it quadratic space} is a pair $(Q,\,V)$ where $V$ is a finite dimensional vector space over $\F_q$ and $Q:\,V\mapsto\F_q$ satisfies that: (1) $Q(\lambda v)=\lambda^2Q(v)$ for all $\lambda\in\F_q$ and $v\in V$; (2) $B_Q(x,y):=Q(x+y)-Q(x)-Q(y)$ is a bilinear form. If we fix a basis $\{e_1,\cdots,e_n\}$ of $V$ over $\F_q$, then $f(x_1,\cdots, x_n)=Q(\sum_{i=1}^nx_ie_i)$ is a {\it quadratic form} in $n$ indeterminants over $\F_q$. A different choice of basis yields an equivalent form, and the {\it rank} of $Q$ is the minimum number of variables in a quadratic form induced from $Q$ in this way. The {\it radical} of a quadratic space $(Q,\,V)$ is
\[
\rad(Q):=\{v\in V:\,B_Q(v,x)=0 \textup{ for all } x\in V\}.
\]
In the following theorem we collect the facts about quadratic forms with $q$ even that we shall need, cf. \cite[Theorem 7.2.9]{handbookff} and Section 6.2 of \cite{ff}.
\begin{thm}\label{thm_N0}
Let $q$ be even and $(Q,V)$ be a quadratic space over $\F_q$. Fix an element $d\in\F_q$ such that $X^2+X+d$ is irreducible over $\F_q$. Write $n:=\dim(V)$, $r=\dim(\rad(Q))$, and denote by $N_0$ the number of $v\in V$ such that $Q(v)=0$. Then $n-r=2s$ is even and there is a basis of $V$ such that the resulting quadratic form is one of the following:
\begin{enumerate}
\item[(i)]  $x_1x_2+x_3x_4+\cdots+x_{2s-1}x_{2s}$ (hyperbolic),
\item[(ii)]  $x_1^2+x_1x_2+dx_2^2+x_3x_4+\cdots+x_{2s-1}x_{2s}$ (elliptic),
\item[(iii)] $x_0^2+x_1x_2+\cdots+x_{2s-1}x_{2s}$ (parabolic),
\end{enumerate}
and $N_0=q^{n-1}+(q-1)q^{r+s-1}\epsilon$, where $\epsilon=1,\,-1,\,0$ in (i), (ii) and (iii) respectively. Moreover, (iii) occurs if and only if $Q(\rad(Q))\ne \{0\}$.
\end{thm}
\begin{corollary}\label{cor_rank}
Let $q$ be even and $(Q,V)$ be a quadratic space over $\F_q$ with $\dim_{\F_q}V=n$. Then the number of $v\in V$ such that $Q(v)=0$ is $q^{n-1}$ if and only if $Q$ has odd rank.
\end{corollary}
\begin{proof}
In Theorem \ref{thm_N0}, $Q$ has even rank in case (i) and (ii), and has odd rank in case (iii).
\end{proof}

For $n\ge 1$, we define the trace function $\tr_{\F_{q^n}/\F_q}:\,\F_{q^n}\mapsto\F_q$ by $\tr_{\F_{q^n}/\F_q}(x)=x+x^q+\cdots+x^{q^{n-1}}$. If a function $f:\,\F_{q^n}\mapsto\F_{q^n}$ satisfies that  $f(\lambda x)=\lambda^df(x)$ for some positive integer $d$ and all $\lambda\in\F_q$ and $x\in\F_{q^n}$, then we call $f(X)$ a {\it homogeneous} polynomial of {\it degree} $d$ over the subfield $\F_q$.  Such a polynomial $f$ naturally induces a map $\bar{f}:\,\F_{q^n}^*/\F_q^*\mapsto \F_{q^n}^*/\F_q^*$ such that $\bar{f}(\bar{x}):=\overline{f(x)}$ for each $\bar{x}=x\F_q^*$. In the case $\bar{f}$ is a bijection, we say that $f$ induces a permutation of $\F_{q^n}^*/\F_q^*$ or simply say that $f$  permutes $\F_{q^n}^*/\F_q^*$.

\begin{lemma}\label{lemma_perm_rank} Let $Q(X)=\sum_{i,j}a_{ij}X^{q^i+q^j}\in\F_{q^n}[X]$ with $q$ even. Then $Q(X)$ is a permutation polynomial of $\F_{q^n}$ if and only if $Q_y(x)=\tr_{\F_{q^n}/\F_q}(yQ(x))$ has odd rank for $y\ne 0$.
\end{lemma}
\begin{proof}
If $Q(X)$ is a permutation polynomial of $\F_{q^n}$, then the size of $\{x:\,Q_y(x)=0\}$ is equal to that of $\{z:\,\tr_{\F_{q^n}/\F_q}(yz)=0\}$, which is $q^{n-1}$ for $y\ne 0$. By Corollary \ref{cor_rank}, the quadratic form $Q_y$ has odd rank for each $y\ne 0$. This proves the necessary part.

Now assume that $Q_y$ has odd rank for each $y\ne 0$. Since $q$ is even and $Q(X)$ is homogeneous of degree $2$ over $\F_q$, it is easy to show that $Q(X)$ is a permutation polynomial if and only if it induces a permutation of the quotient group $\F_{q^n}^*/\F_q^*$. We regard  $\F_{q^n}$ as a vector space over $\F_q$, and identify the elements of  $\F_{q^n}^*/\F_q^*$ with the projective points of $\PG(n-1,q)$ in the natural way. Each hyperplane of $\PG(n-1,q)$ is of the form $\{x\F_{q}^*:\,\tr_{\F_{q^n}/\F_{q}}(yx )=0\}$ for some $y\ne 0$, and $Q(x)\F_q^*$ lies on the hyperplane if and only if $Q_y(x)=0$. For each element $g\in \F_{q^n}^*/\F_{q}^*$, we define $n(g):=\{x\in \F_{q^n}^*/\F_{q}^*:\, Q(x)\F_q^*=g\}$.
For a set $U$ of projective points, write $n(U):=\sum_{g\in U}n(g)$. If $H$ is a hyperplane, then $n(H)=\frac{q^{n-1}-1}{q-1}$ by Corollary \ref{cor_rank}.
For a fixed point $g\in\PG(n-1,q)$, let $H_1,\cdots, H_{\frac{q^{n-1}-1}{q-1}}$ be the set of hyperplanes containing it. As a multiset, $\bigcup_iH_i$ covers the point $g$ exactly $\frac{q^{n-1}-1}{q-1}$ times and each other point $\frac{q^{n-2}-1}{q-1}$ times. We thus have
\begin{align*}
q^{n-2}\cdot n(g)+\frac{q^{n-2}-1}{q-1}\cdot\sum_{h\in \PG(n-1,q)}n(h)=\sum_in(H_i)=\left(\frac{q^{n-1}-1}{q-1}\right)^2.
\end{align*}
On the other hand, $\sum_{h\in \PG(n-1,q)}n(h)=\frac{q^n-1}{q-1}$. It follows that $n(g)=1$ for each $g$. Hence $Q(X)$ induces a permutation of $\F_{q^n}^*/\F_q^*$, and so it permutes $\F_{q^n}$. This proves the sufficiency part.
\end{proof}
The above lemma describes how to study the permutation behavior of a polynomial of DO type via quadratic forms. This is the technique that we will apply in Section 5.

\section{A function approach to $\cC$-planes and $\cH$-planes}\label{sect_function}

Let $q$ be a prime power, $n\ge 2$ be an integer, and regard   the finite field $\F_{q^{2n}}$  as a $2n$-dimensional vector space over $\F_q$. Let $\gamma$ be a primitive element of $\F_{q^{2n}}$ and let $\sigma:\,x\mapsto x^p$ be the Frobenius automorphism. For each $a\in\F_{q^{2n}}^*$, we use $\Theta(a)$ for the linear map from $\F_{q^{2n}}$ to itself that maps $x$ to $ax$. The group $\la\Theta(\gamma)\ra$ is a Singer group, and $\Gamma L(1,q^{2n})=\la\Theta(\gamma)\ra\rtimes \Aut(\F_{q^{2n}})$. Let $\cS$ be a spread with ambient space $(\F_{q^{2n}},+)$ and kernel $\F_q$. The associated affine plane $\Pi(\cS)$ is flag-transitive if and only if $\Aut(\cS)$ is transitive on the components. If further $\Aut(\cS)$ is solvable, then it is isomorphic to a subgroup of $\Gamma L(1,q^{2n})$. After taking an isomorphic spread if necessary, we assume that  $\Aut(\cS)$ is a subgroup of $\la\Theta(\gamma)\ra\rtimes \Aut(\F_{q^{2n}})$.
If the Singer subgroup $\Aut(\cS)\cap \la\Theta(\gamma)\ra$ has order $\frac{q^n+1}{2}$  or $q^n+1$ respectively, then we call $\cS$ a spread of {\it type $\cH$} or {\it type $\cC$} and the corresponding plane  $\Pi(\cS)$ a {\it $\cH$-plane} or a $\cC$-plane respectively. All the known flag-transitive affine planes with a solvable full collineation group  are either $\cC$-planes or $\cH$-planes, and there are no other such planes of order at most $125$. The following result summarizes Lemma 1 and the subsequent comments in \cite{ebert_sur}, restricting to the solvable case.

\begin{lemma}\label{lemma_gcd}
Let $q=p^e$ with $p$  prime and $n\ge2$ be an integer. If $\cS$ is a spread of $(\F_{q^{2n}},+)$ with kernel $\F_q$ such that $\Aut(\cS)$ is solvable and transitive on the components, then $\cS$ is either of type $\cH$ or of type $\cC$ provided that $\gcd(\frac{q^n+1}{2},ne)=1$ in the case $p$ is odd and $\gcd(q^n+1,ne)=1$ in the case $p=2$.
\end{lemma}

The above lemma indicates that $\cC$-planes and $\cH$-planes are the only possibilities under a mild number theoretical condition in the solvable case, and it is an open problem whether the $\gcd$ condition can be dropped. In this paper, we will focus on $\cC$-planes and $\cH$-planes. {\it Throughout the paper, we fix the following notation}. Take  $\beta$ to be an element of order $(q^n+1)(q-1)$. Let $\cS$ be a spread of type $\cH$ or type $\cC$ such that $\Aut(\cS)\cap \la\Theta(\gamma)\ra=\la\Theta(\beta^2)\ra$ or $\la\Theta(\beta)\ra$ respectively. Let $W$ be a component of $\cS$, so that $\cS=\{g(W):
\,g\in \Aut(\cS)\}$. Since the number of $\F_{q^n}$-subspaces of $\F_{q^{2n}}$ is $q^n+1$, there exists $\delta\in\F_{q^{2n}}\setminus\F_{q^n}$ such that $W\cap\F_{q^n}\cdot\delta=\{0\}$. From the decomposition $\F_{q^{2n}}=\F_{q^n}\oplus\F_{q^n}\cdot\delta$, we can write the $\F_q$-subspace $W$ as follows:
\begin{equation}\label{eqn_W}
W=\{x+\delta\cdot L(x):\,x\in\F_{q^n}\},
\end{equation}
where $L(X)\in\F_{q^n}[X]$ is a reduced $q$-polynomial. We also define
\begin{equation}\label{eqn_Q}Q(X):=(X+\delta L(X))\cdot(X+\delta^{q^n} L(X)),
\end{equation}
which is a DO polynomial over $\F_{q^n}$. The following is our key lemma.

\begin{lemma}\label{lemma_key} Take notation as above, and let $W$ be the $\F_q$-subspace  in \eqref{eqn_W}.
\begin{enumerate}
\item If $q$ is odd, then the orbit of $W$ under the group  $\la\Theta(\beta^2)\ra$ forms a partial spread if and only if $Q(x)$ is a planar function over $\F_{q^n}$.
\item If $q$ is odd, then the orbit of $W$ under the group $\la\Theta(\beta)\ra$ forms a spread if and only if $x\mapsto Q(x)$ induces a permutation of $\F_{q^n}^*/\F_q^*$.
\item If $q$ is even, then the orbit of $W$ under the group $\la\Theta(\beta)\ra$ forms a spread if and only if $x\mapsto Q(x)$ is a permutation of $\F_{q^n}$.
\end{enumerate}
\end{lemma}
\begin{proof}
We first prove (1). The orbit of $W$ under $\la\Theta(\beta^2)\ra$ forms a partial spread if and only if the following holds: $y+\delta L(y)=\beta^{2i}(x+\delta L(x))\ne 0$ occurs only in the case $\beta^{2i}\in\F_q^*$ and $y=\beta^{2i}x$.
First assume that we get a partial spread from $W$ as described. If $Q(x)=Q(y)$ for $xy\ne 0$, then  $s_1^{q^n+1}=1$, where $s_1=\frac{y+\delta L(y)}{x+\delta L(x)}$. Since $s_1\in\la \beta^2\ra$,  it follows from our assumption that $s_1\in\F_q^*$ and $y=s_1x$. Since $\gcd(q^n+1,q-1)=2$, we get $s_1^2=1$, i.e., $s_1=\pm 1$. We thus have shown that $x\mapsto Q(x)$ is $2$-to-$1$, and it follows that $Q(x)$ is planar by Lemma \ref{lemma_2to1}.
Conversely, assume that $Q(x)$ is a planar function. If $y+\delta L(y)=\beta^{2i}(x+\delta L(x))\ne 0$, then taking norm we get $Q(y)=\beta^{2(q^n+1)i}Q(x)=Q(\beta^{(q^n+1)i}x)$. It follows from  Lemma \ref{lemma_2to1} that $y=\pm \beta^{(q^n+1)i}x$. Plugging this into $y+\delta L(y)$, we get $\pm \beta^{(q^n+1)i}=\beta^{2i}$. This gives that $\beta^{2i(q^n-1)}=1$, i.e., $\beta^{2i}\in\F_{q^n}^*$. It is easy to show that $\la \beta^2\ra\cap \F_{q^n}^*=\F_q^*$, so $\beta^{2i}\in\F_q^*$. The conclusion now follows.

We next prove (2). Since $Q(\lambda x)=\lambda^2Q(x)$ for $\lambda\in\F_q$, the map $x\mapsto Q(x)$ induces a function from $\F_{q^n}^*/\F_q^*$ to itself. Observe that $\la\beta\ra$ is the set of elements in $\F_{q^{2n}}^*$ whose relative norm to $\F_{q^n}$ is in $\F_q^*$. The orbit of $W$ under $\la\Theta(\beta)\ra$ forms a spread if and only if the following holds: $y+\delta L(y)=\beta^{i}(x+\delta L(x))\ne 0$ occurs only in the case $\beta^{i}\in\F_q^*$ and $y=\beta^{i}x$.
First assume that we get a  spread from $W$ as described. If $Q(x)/Q(y)\in\F_{q}^*$ for $xy\ne 0$, then $y+\delta L(y)=s_1(x+\delta L(x))$ for some $s_1\in\la\beta\ra$. It follows that $s_1\in\F_q^*$ and $y=s_1x$. Hence $Q(x)$ permutes $\F_{q^n}^*/\F_q^*$.
Conversely, assume that $Q(x)$ permutes $\F_{q^n}^*/\F_q^*$. If $y+\delta L(y)=\beta^{i}(x+\delta L(x))\ne 0$, then taking norm we get $Q(y)/Q(x)=\beta^{i(1+q^n)}\in\F_{q}^*$. It follows that $y/x\in\F_q^*$, and  the conclusion follows.

The claim (3) can be proved similarly, and we omit the details.
\end{proof}
\begin{corollary}
Let $W$ be the $\F_q$-subspace in \eqref{eqn_W}, and assume that $q$ and $n$ are odd. If the orbit of $W$ under $\la\Theta(\beta^2)\ra$ forms a partial spread, then its orbit under $\la\Theta(\beta)\ra$ forms a  spread of type $\cC$.
\end{corollary}
\begin{proof}
By the previous lemma, we see that $Q(x)$ is a planar function, and so $x\mapsto Q(x)$ is $2$-to-$1$ by Lemma \ref{lemma_2to1}. Denote by $D$ the image set  $\{Q(x):\,x\ne 0\}$, and write $E$ for its complement in $\F_{q^n}^*$. For a nonsquare $\lambda\in\F_q^*$, we have $E=\lambda \cdot D$  by \cite[Proposition 3.6]{weng}. In particular, $Q(y)/Q(x)\in\F_q^*$ implies that $Q(y)/Q(x)=u^2$ for some $u\in\F_q^*$. The $2$-to-$1$ property of $Q$ then gives that  $y/x=\pm u\in\F_q^*$. This shows that $Q(x)$ permutes $\F_{q^n}^*/\F_q^*$. The claim then follows from claim (2) in Lemma \ref{lemma_key}.
\end{proof}

\begin{thm}
There is no type $\mathcal{C}$ spread with ambient space $(\F_{q^{2n}},+)$ and kernel $\F_q$ when $n$ is even and $q$ is odd.
\end{thm}
\begin{proof}
We take notation introduced preceding Lemma \ref{lemma_key} and prove by contradiction. By claim (2) in Lemma \ref{lemma_key}, $x\mapsto Q(x)$ induces a permutation of $\F_{q^n}^*/\F_q^*$.
By \cite[Lemma 3.3]{weng}, $\tr_{\F_{q^n}/\F_q}(Q(x))$ is a nondegenerate quadratic form over $\F_q$. By \cite[Theorem 6.26]{ff},
the number $N_0=\#\{x\in\F_{q^n}:\,\tr_{\F_{q^n}/\F_q}(Q(x))=0\}$ is equal to $q^{n-1}\pm(q-1)q^{n/2-1}$. On the other hand, $N_0=\#\{y\in\F_{q^n}:\,\tr_{\F_{q^n}/\F_q}(y)=0\}=q^{n-1}$ by the fact that  $Q(x)$ permutes $\F_{q^n}^*/\F_q^*$. This contradiction completes the proof.
\end{proof}

In the case $q$ is odd, we may further restrict the form of $W$.
\begin{lemma}\label{lemma_delta}
Let $\cS$ be a spread  such that $\Aut(\cS)\cap \la\Theta(\gamma)\ra$ contains $\la\Theta(\beta^2)\ra$, and let $W$ be a component.
If $q$ is odd and $\delta^{q^n-1}=-1$, then $W$ intersects $\F_{q^n}$ or $\F_{q^n}\cdot\delta$ trivially.
\end{lemma}
\begin{proof}
From the decomposition $\F_{q^{2n}}=\F_{q^n}\oplus\F_{q^n}\cdot\delta$, we can write $W=\{L_1(x)+\delta L_2(x):\,x\in\F_{q^n}\}$ for some $q^n$-polynomials $L_1(X),\,L_2(X)$. The argument in Lemma \ref{lemma_key} shows that $Q(x)=L_1(x)^2-\delta^2L_2(x)^2$ is planar. By Lemma \ref{lemma_L_perm}, at least one of $L_1,\,L_2$ is a permutation, and correspondingly $W$ intersects one of $\F_{q^n}$ and $\F_{q^n}\cdot\delta$ trivially.
\end{proof}

\begin{remark}In Lemma \ref{lemma_delta},  if $\Aut(\cS)$ is transitive, then $\cS'=\{g(\delta\cdot W):\,g\in \Theta(\delta)\cdot\Aut(\cS)\cdot\Theta(\delta)^{-1}\}$ is a spread isomorphic to $\cS$ and $\Aut(\cS')$ contains $\la\Theta(\beta^2)\ra$. So after replacing $W$ by $\delta\cdot W$ if necessary, we always assume that $\delta^{q^n-1}=-1$ in \eqref{eqn_W} in the case $q$ is odd.
\end{remark}

The Kantor-Suetake family constitutes a major part of the known non-Desarguesian flag-transitive affine planes of odd order.  As a first application of our new approach, we give a characterization of this important family based on the following  result of Menichetti. Please refer to \cite{albert_tw,handbook} for details on generalized twisted fields.
\begin{thm}\label{thm_men}\cite{meni1,meni2}
Let $S$ be a finite semifield of prime dimension $n$ over the nucleus $\F_q$. Then there is an integer $\nu(n)$ such that if $q\ge \nu(n)$ then $S$ is isotopic to a finite field or a generalized twisted field. Moreover, we have $\nu(3)=0$.
\end{thm}
By \cite[Proposition 11.31]{handbook}, which is essentially due to Albert \cite{albert_iso1,albert_iso2}, a generalized twisted field that has a commutative isotope must be isotopic to the commutative presemifield defined by a planar function $x^{1+p^\alpha}$ over $\F_{p^{e}}$, where $1\le \alpha \le n-1$ and $e/\gcd(e,\alpha)$ is odd.  The following result characterizes planar functions whose associated presemifield is isotopic to a commutative twisted field or a finite field, see \cite[Corollaries 3.9, 3.10]{coulter}.

\begin{lemma}\label{lemma_iso} Let $p$ be an odd prime and $q=p^e$. Let $f$ be a planar function of DO type over $\F_q$ and $S_f=(\F_q,+,*)$ be the associated presemifield with $x*y=f(x+y)-f(x)-f(y)$. There exist linearized permutation polynomials $M_1$ and $M_2$ such that
\begin{enumerate}
\item if $S_f$ is isotopic to a finite field, then $f(M_2(x))\equiv M_1(x^2)$ for $x\in\F_{q}$;
\item if $S_f$ is isotopic to a commutative twisted field, then $f(M_2(x))=M_1(x^{p^\alpha+1})$ for $x\in\F_{q}$, where $\alpha$ is an integer such that $1\le\alpha\le e-1$ and $e/\gcd(e,\alpha)$ is odd.
\end{enumerate}
\end{lemma}

\begin{lemma}\label{lemma_ps} Let  $n$ be an odd prime, $\nu(n)$ be as in Theorem \ref{thm_men}, and let $q$ be an odd prime power such that $q\ge \nu(n)$. Suppose that the orbit of $W=\{x+\delta\cdot L(x):\,x\in\F_{q^n}\}$ under $\la\Theta(\beta^2)\ra$ forms a partial spread, where $\delta^{q^n-1}=-1$ and $L(X)$ is a $q$-polynomial over $\F_{q^n}$. Then $W=\{\alpha\cdot(x+u\delta x^{q^i}):\, x\in\F_{q^n}\}$ for some $\alpha\in\F_{q^{2n}}^*$, $u\in\F_{q^n}$ and $0\le i\le n-1$.
\end{lemma}
\begin{proof}
By Lemma \ref{lemma_key}, $Q(x)=x^2-\delta^2L(x)^2$ is a planar function. By Theorem \ref{thm_men}, the associated presemifield $S_Q$ is isotopic to a finite field or a commutative twisted field under the conditions in the lemma. By Lemma \ref{lemma_iso}, there are reduced linearized permutation polynomials $M_1,M_2$ such that
\begin{equation}\label{eqn_M12ff}
M_1(X^2)\equiv M_2(X)^2-\delta^2 L(M_2(X))^2 \pmod{X^{q^n}-X},
\end{equation}
or
\begin{equation}\label{eqn_M12tw}
M_1(X^{1+p^\alpha})\equiv M_2(X)^2-\delta^2 L(M_2(X))^2 \pmod{X^{q^n}-X},
\end{equation}
where $q=p^e$ with $p$ prime, $1\le\alpha\le ne-1$ and $\frac{ne}{\gcd(ne,\alpha)}$ is odd. Write $M_2(X)=\sum_{i=0}^{ne-1}a_iX^{p^i}$, $L(M_2(X))\equiv\sum_{i=0}^{ne-1}b_iX^{p^i}\pmod{X^{q^n}-X}$, and set $I=\{i:\,a_i\ne 0\}$, $J=\{i:\, b_i\ne 0\}$.

Comparing coefficients of $x^{2p^i}$ in \eqref{eqn_M12tw}, we see that $0=a_i^2-\delta^2 b_i^2$ for $0\le i\le ne-1$. Since $\delta^2$ is a nonsquare in $\F_{q^n}^*$, all the $a_i$'s and $b_j$'s are zero. This contradicts the assumption that $M_2$ is a permutation polynomial, so \eqref{eqn_M12tw} can not occur. It remains to check \eqref{eqn_M12ff}. We look at the coefficients of $x^{p^i+p^j}$, $0\le i<j\le ne-1$, on both sides and get $a_ia_j=\delta^2 b_ib_j$. If $|I|\ge 2$ or $|J|\ge 2$, then it is easy to deduce that $I=J$. If both $I$ and $J$ have size at most $1$, then $L$ is a monomial and $W$ takes the desired form with $\alpha=1$. We assume that $I=J$ and they have size at least two below. In this case, for any distinct $i,j\in I$, exactly one of $\{a_ib_i^{-1},\,a_jb_j^{-1}\}$ is a square and the other is a nonsquare, since $\delta^2$ is a nonsquare in $\F_{q^n}^*$. This is only possible when $|I|=2$. Therefore,
$M_2(X)=aX^{p^k}+bX^{p^\ell}$, $L(M_2(X))=cX^{p^k}+dX^{p^\ell}$ for some $0\le k<\ell\le ne-1$ and $a,b,c,d\in\F_{q^n}^*$ such that $ab=\delta^2 cd$. Since $x\mapsto M_2(x)$ is a permutation, the elements of $W$ can be written as
\[
M_2(x)+\delta\cdot L(M_2(x))=(a+\delta c)\cdot(y+\delta da^{-1}y^{p^{\ell-k}}),\;y=x^{p^k}.
\]
Since $W$ is a $\F_q$-linear subspace, $p^{\ell-k}$ is a power of $q$.
This completes the proof.
\end{proof}

\begin{remark}\label{remark_size} Let $q$ be an odd prime power, $w$ be a nonsquare of $\F_q$, and $N>1$ be an integer. Let ${\bf a}=(a_0,a_1\cdots,a_{N-1})$ and ${\bf b}=(b_0,b_1,\cdots,b_{N-1})$ be two sequences consisting of elements in $\F_q$, and define their supports as $I_1=\{i:\,a_i\ne 0\}$, $I_2=\{i:\,b_i\ne 0\}$ respectively. As in the proof of Lemma \ref{lemma_ps}, we can show that: if   $a_ia_j=wb_ib_j$ for any distinct $i,\,j$, then either both $|I_1|$ and $|I_2|$ have size at most one or $I_1=I_2$ and both have size $2$.
\end{remark}

\begin{thm}\label{thm_typeC}
Let  $n$ be an odd prime, $\nu(n)$ be as in Theorem \ref{thm_men}, and let $q$ be an odd prime power such that $q\ge \nu(n)$. A type $\cC$ spread $\cS$ of $ (\F_{q^{2n}},+)$ with kernel $\F_q$ is isomorphic to the orbit of
$W=\{x+\delta\cdot x^{q^i}:\,x\in\F_{q^n}\}$ under $\la\Theta(\beta)\ra$ for some $\delta$ and $i$ such that $\delta^{q^n-1}=-1$, $1\le i\le n-1$ and $\gcd(i,n)=1$.
\end{thm}
\begin{proof}
Let $W$ be a component of $\cS$ such that $\cS$ is the orbit of $W$ under the Singer subgroup $\la\Theta(\beta)\ra$. By the remark following Lemma \ref{lemma_delta}, up to isomorphism $W=\{x+\delta' L(x):\,x\in\F_{q^n}\}$ for some $q$-polynomial $L(X)$ and $\delta'$ such that $\delta'^{q^n-1}=-1$. By Lemma \ref{lemma_ps}, $W=\{\alpha\cdot(x+u\delta' x^{q^i}):\, x\in\F_{q^n}\}$ for some $\alpha\in\F_{q^{2n}}^*$ and $u\in\F_{q^n}^*$. Its orbit under $\la\Theta(\beta)\ra$ is a spread  isomorphic to the one described in the theorem with $\delta=u\delta'$. The kernel contains the fixed subfield of $x\mapsto x^{q^i}$, so we have $\gcd(i,n)=1$.
\end{proof}

\begin{thm}\label{thm_typeH}
Let  $n$ be an odd prime, $\nu(n)$ be as in Theorem \ref{thm_men}, and let $q$ be an odd prime power such that $q\ge \nu(n)$. A type $\cH$ spread $\cS$ of $(\F_{q^{2n}},+)$ with kernel $\F_q$ is isomorphic to the orbit of $\{x+\delta x^{q^k}:\,x\in\F_{q^n}\}$ under the group $A$ generated by $\la\Theta(\beta^2)\ra$ and $\psi:\,z\mapsto \eta z^{q^n}$, where $1\le k\le n-1$, $\gcd(k,n)=1$, $\delta^{q^n-1}=-1$,  $\eta^{(1+q^n)(q^k-1)}=1$ and $\eta$ is a nonsquare.
\end{thm}
\begin{proof} Write $q=p^e$ with $p$ prime. By Lemma \ref{lemma_ps}, under the conditions in the theorem a spread $\cS$ of type $\cH$  is  isomorphic to the orbit of $W=\{x+\delta\cdot x^{\tau}:\,x\in\F_{q^n}\}$ under $\Aut(\cS)$, where $\delta^{q^n-1}=-1$, $\tau=q^k$ with $1\le k\le n-1$ and $\gcd(k,n)=1$, and $\Aut(\cS)\cap\la\Theta(\gamma)\ra=\la\Theta(\beta^2)\ra$.
There exists  $\psi\in\Aut(\cS)$ that permutes the two $\la\Theta(\beta^2)\ra$-orbits by the  transitivity of $\Aut(\cS)$ on the components, so the spread $\cS$ consists of the orbit of $W$ under the subgroup $A:=\la\Theta(\beta^2),\,\psi\ra$. By elementary group theory,
\[
\Aut(\cS)/\la\Theta(\beta^2)\ra\cong \Aut(\cS)\cdot\la\Theta(\gamma)\ra/\la\Theta(\gamma)\ra\le
\la\Theta(\gamma)\ra\rtimes \Aut(\F_{q^{2n}})/\la\Theta(\gamma)\ra
\cong \Aut(\F_{q^{2n}}).
\]
Since any odd power of $\psi$ also permutes the two $\la\Theta(\beta^2)\ra$-orbits, we can assume that the order of $\overline{\psi}\in\Aut(\cS)/\la\Theta(\beta^2)$ is a power of $2$. Write $\psi(z)=\eta z^\sigma$, where $\sigma=p^\ell$ ($1\le \ell\le 2ne-1$) and $\eta\in\F_{q^{2n}}^*$. We now derive conditions on $\psi$ to guarantee that the orbit of $W$ under the subgroup $A$ forms a spread. To be specific, we need to make sure the following hold:
\begin{enumerate}
\item $\psi^2(W)=\beta^{2i}\cdot W$ for some $i$, since $\psi^2$ stabilizes the $\la\Theta(\beta^2)\ra$-orbits;
\item $\psi(W)$ intersects each of $\beta^{2i}\cdot W$, $0\le i\le \frac{q^n-1}{2}$, trivially.
\end{enumerate}

We fist consider the condition (1). Assume that $\psi^2(W)=\beta^{2i}\cdot W$ for some $i$. This means that for each $x\in\F_{q^n}$, there exists $y\in\F_{q^n}$ such that
$\eta^{1+\sigma}\left(x+\delta x^\tau\right)^{\sigma^2}=\beta^{2i}(y+\delta y^\tau).$
Write $\eta^{1+\sigma}\beta^{-2i}=u+v\delta$ for some $u,v\in\F_{q^n}$. Expanding and comparing the coefficients of the basis $\{1,\delta\}$, we get
$y=ux^{ \sigma^2}+v\delta^{\sigma^2+1}x^{\tau\sigma^2}$, $ y^\tau=u\delta^{\sigma^2-1}x^{\tau\sigma^2}+vx^{ \sigma^2}$.
Canceling out $y$, we see that $(u^\tau-u\delta^{\sigma^2-1})x^{\tau\sigma^2}+v^\tau\delta^{\tau(\sigma^2+1)}x^{\tau^2\sigma^2}
-vx^{\sigma^2}=0$ for all $x$. Hence, the reduced polynomial
$(u^\tau-u\delta^{\sigma^2-1})X^{\tau\sigma^2}+v^\tau\delta^{\tau(\sigma^2+1)}X^{\tau^2\sigma^2}
-vX^{ \sigma^2}\pmod{X^{q^n}-X}$ is the zero polynomial.
Since $\tau\ne 1$ and $n$ is odd, the reduced monomial $X^{\sigma^2}\pmod{X^{q^n}-X}$ occurs only once.  This gives $v=0$, and the remaining coefficient gives $u^{\tau-1}=\delta^{\sigma^2-1}$. Raising both sides of $\eta^{1+\sigma}\beta^{-2i}=u$ to the $\frac{1}{2}(q^n+1)(\tau-1)$-st power, we get  $\eta^{(q^n+1)(\tau-1)(1+\sigma)/2}=u^{\tau-1}$, so it must hold that
\begin{equation}\label{eqn_eta_del}
\eta^{(1+q^n)(\tau-1)(1+\sigma)/2}=\delta^{\sigma^2-1}.
\end{equation}
Conversely, if \eqref{eqn_eta_del} is true, then (1) holds with  $\beta^{2i}=\eta^{ (1-q^n) (1+\sigma)/2}$ by direct check. To summarize, we have shown that  (1) holds if and only if \eqref{eqn_eta_del} holds.

We claim that  $\gcd(e,\ell)=r$, $\delta_0^{(p^r-1)/2}=-1$ and $\eta^{(1+q^n)(\tau-1)}=\delta^{2(\sigma-1)}$, where $r=\gcd(e,2\ell)$, and $\delta_0=\delta^{2(q^n-1)/(q-1)}$.  By raising both sides of  \eqref{eqn_eta_del} to the $\left(\frac{q^n-1}{q-1}\right)$-th power we get  $\delta_0^{(p^{2\ell}-1)/2}=1$. It follows from $\gcd(p^{2\ell}-1,p^e-1)=p^r-1$ that $\delta_0^{p^r-1}=1$. On the other hand, $\delta^2$ is a nonsquare in $\F_{q^n}^*$, so $\delta_0^{(p^e-1)/2}=-1$.
Now, we see that $\delta_0$ has order dividing $\gcd(\frac{p^{2\ell}-1}{2},p^r-1)$ which is equal to either $p^r-1$ or $(p^r-1)/2$. The latter case will not occur, since it would lead to the contradiction $\delta_0^{(p^e-1)/2}=1$. Hence $p^r-1$ divides $\frac{p^{2\ell}-1}{2}$, which is the case only if $2\ell/r$ is even, i.e., $r|\ell$. This shows that $r=\gcd(e,\ell)$ and $\delta_0^{(p^r-1)/2}=-1$. It is clear that $g:=\gcd\left(\frac{q^n-1}{q-1},\frac{\sigma+1}{2}\right)$ divides $\gcd(p^{ne}-1,p^{2\ell n}-1)=p^{\gcd(ne,2\ell n)}-1=p^{nr}-1$, so it also divides $\gcd(p^{n\ell}-1,p^\ell+1)=2$. On the other hand, $\frac{q^n-1}{q-1}$ is odd, so $g=1$. It follows from \eqref{eqn_eta_del} that $\eta^{(1+q^n)(\tau-1)}=\delta^{2(\sigma-1)}$.

We now consider the condition (2). Recall that $\la\beta\ra$ is the set of elements of $\F_{q^{2n}}^*$ whose relative norm to $\F_{q^n}$ is in $\F_q^*$. The condition amounts to that $\eta^{1+q^n}Q(x)^\sigma=\lambda^2Q(y)$ does not hold for any $x,y\in\F_{q^n}^*$ and $\lambda\in\F_q^*$, where $Q$ is as defined in \eqref{eqn_Q}. By expanding $\eta^{1+q^n}Q(x)^\sigma=\lambda^2Q(y)$ and rearranging terms, we get $Y=\delta^2Y^\tau$ with $Y=\eta^{1+q^n}x^{2\sigma}-\lambda^2y^2\in\F_{q^n}$. Here we have made use of the fact that $\eta^{(1+q^n)(\tau-1)}=\delta^{2(\sigma-1)}$. If $Y\ne 0$, then $Y^{1-\tau}$ is  a square  while $\delta^2$ is not, which is impossible. We thus must have $Y=0$. It is clear that $Y=0$ has a  solution  $(x,y,\lambda)\in\F_{q^n}^*\times\F_{q^n}^*\times\F_q^*$ if and only if $\eta^{1+q^n}$ is a square in $\F_{q^n}^*$. To summarize,  (2) holds only if $\eta^{1+q^n}$ is a nonsquare in $\F_{q^n}^*$.

Let $e'$ be the highest power of $2$ that divides $e$. It follows from $r=\gcd(e,\ell)=\gcd(e,2\ell)$ that $e/r$ is odd, i.e., $e'|r$. Hence $\ell$ is a multiple of $e'$. Recall that the order of $\overline{\psi}$ is a power of $2$, so $\ell$ is a multiple of $ne$. On the other hand, $0<\ell<2ne$ implies that $\ell=ne$. In this case, $r=e$, and the conditions reduce to the same as stated in the theorem.

The sufficiency part of the theorem is shown in \cite{kantor_odd,ks}.
\end{proof}

The spreads described in Theorem \ref{thm_typeC} and Theorem \ref{thm_typeH} are due to Kantor and Suetake \cite{kantor_odd,ks}. The case $n=3$ has been characterized by Baker and Ebert et al in a series of papers \cite{be_last,be_baer2,be_baer,be_x} by a geometric approach. It is well-known that a two-dimensional finite semifield is a finite field   \cite{dickson_2dim}, so the same arguments in this section can be applied to characterize the case $n=2$ which has been dealt with in \cite{be_2dim}.

\section{Planar functions of the form $L(x)^2-wx^2$}

In Section \ref{sect_function}, we have shown that $\cC$-planes and $\cH$-planes of odd order have close connections with planar functions of the form $X^2-\delta^2 L(X)^2$ over $\F_{q^n}$, where $\delta^{q^n-1}=-1$. In this section, we study the properties of the commutative semifields associated with such planar functions and determine the planarity of functions of particular forms. Throughout this section, we shall fix the following notation. For convenience, we write alternatively $Q(X)=L(X)^2-wX^2$, where $w(=\delta^{-2})$ is a nonsquare in $\F_{q^n}$ and $L(X)$ is a reduced $q$-polynomial such that
\begin{equation}\label{eqn_L_ker}
\{\lambda\in\F_{q^n}:\,L(\lambda x)=\lambda L(x) \textup{ for all } x \in\F_{q^n}\}=\F_q,
\end{equation}
Assume that $Q(X)$ is a planar function over $\F_{q^n}$. Then
\[
x\mapsto\,\frac{1}{2}\left(Q(x+1)-Q(x)-Q(1)\right)=L(1)L(x)-wx
\]
is a $\F_q$-linear bijection of $\F_{q^n}$ by definition. Let $L_1$ be its inverse under composition, which is also $\F_q$-linear. Denote by $S_Q=(\F_{q^n},+,\circ)$ the associated semifield with multiplication
\begin{equation}\label{eqn_circ}
x\circ y=L(L_1(x))L(L_1(y))-wL_1(x)L_1(y),
\end{equation}
The multiplicative identity is $e_Q=L(1)^2-w$, and $\cN(S_Q)$ contains $\F_q\cdot e_Q$.  We have  $(ae_Q)\circ z=az$ for $a\in \F_q$ and $z\in\F_{q^n}$. Let $q^m$ and $q^r$ be the sizes of $\cN_m(S_Q)$ and $\cN(S_Q)$ respectively.

We now explicitly construct a commutative isotope of $S_Q$ that has $\F_{q^m}$ as the middle nucleus.  For each $x\in S_Q$, let $R_x$ be the $\F_q$-linear map over $(\F_{q^n},+)$ such that $R_x(z)=z\circ x$. For  $u\in \cN_m(S_Q)$ and $i\ge 1$, we use $u^{\circ i}$ for the product of $i$ copies of $u$ under the multiplication $\circ$, and set $u^{\circ 0}:=e_Q$.   Let $z'$ be a fixed primitive element of $\F_{q^m}$ with minimal polynomial $X^m+\sum_{i=0}^{m-1}c_iX^i$ over $\F_q$. The map $a\mapsto ae_Q$ is a field isomorphism between $(\F_q,+,\cdot)$ and $(\F_q\cdot e_Q,+,\circ)$, and it naturally extends to a ring isomorphism between their polynomial rings. Therefore, there exists a primitive element $z\in\cN_m(S_Q)$ such that $z^{\circ m}+\sum_{i=0}^{m-1}(c_ie_Q)\circ z^{\circ i}=0$.  Let $\{f_1=e_Q,f_2\cdots ,f_{n/m}\}$ and $\{f_1'=1,f_2'\cdots ,f_{n/m}'\}$ be a basis of $(\F_{q^n},+)$ over $\cN_m(S_Q)$ and $\F_{q^m}$ respectively. Now define a $\F_q$-linear map $\psi:\,(\F_{q^n},+)\mapsto (\F_{q^n},+)$ such that
\[
\psi(z'^if_j')=z^{\circ i}\circ f_j,\quad 0\le i\le m-1 , \;1\le j\le n/m.
\]
Here, $\F_q$-linearity means that $\psi(\lambda x)=\lambda\psi(x)$ for  $\lambda\in\F_q$ and $x\in\F_{q^n}$.  The map $\psi$ is a bijection and has the properties: (1) The restriction $\psi|_{\F_{q^m}}$ is a field isomorphism between $\F_{q^m}$ and $\cN_m(S_Q)$; (2) $\psi(ax)=\psi(a)\circ \psi(x)$, i.e., $\psi^{-1}R_{\psi(a)}\psi(x)=ax$,  for $a\in\F_{q^m}$ and $x\in\F_{q^n}$. By the definition of the nucleus, $R_{\psi(y)}R_{\psi(a)}=R_{\psi(a)}R_{\psi(y)}$ for  $a\in \F_{q^r}$ and $y\in \F_{q^n}$, so $\psi^{-1}R_{\psi(y)}\psi$ is $\F_{q^r}$-linear for each $y$ by  (2). On the other hand,  $\psi^{-1}(\psi(y)\circ\psi(x))=\psi^{-1}R_{\psi(y)}\psi(x)$ is symmetric in $x,y$, so it is $\F_{q^r}$-linear in both $x$ and $y$. Therefore, $\psi^{-1}(\psi(y)\circ\psi(x))=x\ast_K y$, where $x\ast_K y=\sum_{i,j}c_{ij}x^{q^{ri}}y^{q^{rj}}$ for some constants $c_{ij}$'s such that $c_{ij}=c_{ji}$. It follows that $\psi(x)\circ \psi(y)=\psi(x\ast_Ky)$ and also $a*_Kx=ax$ for   $a\in\F_{q^m}$ by (2). The algebraic system $S_K:=(\F_{q^n},+,\ast_K)$ is a semifield isotopic to $S_Q$, and its  middle nucleus is $\F_{q^m}$.

We write $K(X)=\sum_{i,j=0}^{n/r-1}c_{ij}X^{q^{ri}+q^{rj}}$, so that $K(x)=x\ast_K x$ for $x\in\F_{q^n}$. Let $M_1(X),M_2(X)$ be the reduced $q$-polynomials s.t. $M_1(x)=L(L_1(\psi(x)))$ and $M_2(x)=L_1(\psi(x))$ for $x\in\F_{q^n}$ respectively. Applying the map $x\mapsto L(1)L(x)-wx$ to both sides of $M_2(x)=L_1(\psi(x))$, we get $\psi(x)=L(1)M_1(x)-wM_2(x)$. Now $\psi(x)\circ \psi(x)=\psi(x\ast_Kx)$ takes the form
\begin{equation}\label{eqn_M}
M_1(X)^2-wM_2(X)^2\equiv L(1)M_1(K(X))-wM_2(K(X))\pmod{X^{q^n}-X}.
\end{equation}

\begin{lemma}\label{lemma_N}
Let $q$ be an odd prime power, $n\ge 2$ be an integer, and $w$ be a nonsquare in $\F_{q^n}$. Suppose that $L(X)$ is a reduced $q$-polynomial over $\F_{q^n}$ such that \eqref{eqn_L_ker} holds. If $Q(x)=L(x)^2-wx^2$ is planar over $\F_{q^n}$, then  the semifield $S_Q=(\F_{q^n},+,\circ)$ with multiplication as defined in \eqref{eqn_circ} is either isotopic to a finite field or has nucleus equal to $\F_q$.
\end{lemma}
\begin{proof}
We use the notation introduced preceding the lemma, and write $s=n/r$.  The semifield $S_Q$ is isotopic to a finite field if and only $r=n$, so we assume that $1<r<n$ and try to derive a contradiction. First we introduce some notation for the proof. For a reduced $q$-polynomial $f(X)\in\F_{q^n}[X]$, we have a unique decomposition
\[
f(X)=f_0(X)+f_1(X^{q})+\cdots+f_{r-1}(X^{q^{r-1}}),
\]
where the $f_i$'s are reduced $q^r$-polynomials. We call it the {\it $q^r$-decomposition} of $f$, and call $f_i$ the {\it $i$-th component}. For $t\in\{1,2\}$, write $M_t(X)=\sum_{i=0}^{n-1}a_{i,t}X^{q^i}$, and define
\[
I_t:=\{0\le i\le r-1:\,\textup{one of }a_{i,t},\,a_{i+r,t},\cdots,a_{i+(s-1)r,t}\textup{ is nonzero} \}.
\]
We comment that $i\in I_t$ if and only if $M_t$ has a nonzero $i$-th component in its $q^r$-decomposition. Since $M_2$ is a permutation and $L$ is not the zero polynomial, both  $I_1$ and $I_2$ are nonempty. Observe that for $i\not\equiv j\pmod{r}$, the coefficient of $X^{q^i+q^j}$ is zero on the right hand side of \eqref{eqn_M}, so  we have $a_{i1}a_{j1}=wa_{i2}a_{j2}$ from the left hand side. We can pick two subsequences of length $r$ whose supports are $I_1$ and $I_2$  from the coefficients of $M_1(X)$ and $M_2(X)$ respectively, satisfying the conditions in Remark \ref{remark_size}. It follows that each of $I_1$ and $I_2$ has size at most $2$ and $I_1=I_2$ when one of them has size $2$.

We first consider the case $|I_1|=|I_2|=1$. In this case, we have $M_t(X)=U_t(X^{q^{r_t}})$ for $t\in\{1,2\}$, where $U_1$ and $U_2$ are $q^r$-polynomials, and $0\le r_1,\,r_2\le r-1$.  Set $r_3:=r_1-r_2\pmod{r}$. Since $M_1(X)\equiv L(M_2(X))\pmod{X^{q^n}-X}$, the $q^r$-decomposition of $L(X)$ has exactly one nonzero component, namely the $r_3$-rd.
We have $r_3\ne 0$: otherwise $L(X)$ is  $q^r$-polynomial, contradicting \eqref{eqn_L_ker}. By comparing exponents of the monomials in \eqref{eqn_M},  we get
$M_1(x)^2=L(1)M_1(K(x))$ and $M_2(x)^2=M_2(K(x))$
for  $x\in\F_{q^n}$. Applying $L$ to both sides of the second equation, we get $L(M_2(x)^2)=M_1(K(x))$. Recall that $x\mapsto M_2(x)=L_1(\psi(x))$ is a permutation. By setting $z:=M_2(x)$ and combining $L(M_2(x)^2)=M_1(K(x))$ with the first equation, we get $L(z)^2=L(1)L(z^2)$. It follows that $L(X)$ is a monomial and $S_Q$ is isotopic to a finite field, contradicting our assumption.

We next consider the case $I_1=I_2=\{r_1,r_2\}$, where $0\le r_1<r_2\le r-1$. In this case, we have the $q^r$-decomposition $M_1(X)=A(X^{q^{r_1}})+B(X^{q^{r_2}})$, $M_2(X)=C(X^{q^{r_1}})+D(X^{q^{r_2}})$, where $A,B,C,D$ are reduced $q^r$-polynomials neither of which is the zero polynomial. Plugging them into \eqref{eqn_M} and again by comparing exponents of monomials, we get
\begin{align}
A(X^{q^{r_1}})B(X^{q^{r_2}})&=wC(X^{q^{r_1}})D(X^{q^{r_2}}),\label{eqn_extra}\\
A(X^{q^{r_1}})^2-wC(X^{q^{r_1}})^2&\equiv L(1)A(K(X)^{q^{r_1}})-wC(K(X)^{q^{r_1}})\pmod{X^{q^n}-X},\label{eqn_extra2}\\
B(X^{q^{r_2}})^2-wD(X^{q^{r_2}})^2&\equiv L(1)B(K(X))^{q^{r_2}}-wD(K(X)^{q^{r_2}})\pmod{X^{q^n}-X}.\label{eqn_extra3}
\end{align}
The equation \eqref{eqn_extra} holds without modulo $X^{q^n}-X$ since both sides have degree at most $2q^{n-1}\le q^n-1$. If the coefficient of $X^{q^{r\ell}}$ in $B(X)$ is nonzero, then by considering the coefficients of the monomials $\{X^{q^{r_2+r\ell}+q^{r_1+ri}}: \,0\le i\le s-1\}$ on both sides of \eqref{eqn_extra},  we see that $A$ and $C$ differs by a constant. That is, $A(X)=\lambda C(X)$, $B(X)=\lambda^{-1}wD(X)$ for some $\lambda\ne 0$. Canceling $A$ and $B$ from \eqref{eqn_extra2} and \eqref{eqn_extra3} by substitution, we get
\begin{align}
(\lambda L(1)-w)\cdot C(K(x)^{q^{r_1}})= (\lambda^2-w)\cdot C(x^{q^{r_1}})^2, \label{eqn_CD}\\
(\lambda^2-\lambda L(1))\cdot D(K(x)^{q^{r_2}})=(\lambda^2-w)\cdot D(x^{q^{r_2}})^2.\label{eqn_CD2}
\end{align}
Since $w$ is a nonsquare, $\lambda^2-w\ne 0$. It follows that $\lambda L(1)-w$ and $\lambda^2-\lambda L(1)$ are both nonzero. By expansion using the $q^m$-decomposition of $M_1,\,M_2$, we have
\[
M_1(x)^2-wM_2(x)^2=(\lambda^2-w)\cdot(C(x^{q^{r_1}})^2-\lambda^{-2}wD(x^{q^{r_2}})^2).
\]
The map $x\mapsto Q(M_2(x))=M_1(x)^2-wM_2(x)^2$ is a planar function of DO type equivalent to $Q$, so at least one of $C,\,D$ is a permutation polynomial by Lemma \ref{lemma_L_perm}. If $C$ is a permutation polynomial, then  we substitute $x+y,\,x,\,y$ into \eqref{eqn_CD} and take their linear combination to get  $c_1 C((x\ast_Ky)^{q^{r_1}})=\left(c_1 C(x^{q^{r_1}})\right)\cdot \left(c_1 C(y^{q^{r_1}})\right)$, where $c_1=\frac{\lambda^2-w}{\lambda L(1)-w}$. Here, we have used the fact that $K(x+y)-K(x)-K(y)=2(x*_Ky)$. This shows that $S_K$ is isotopic to a finite field, contradicting our assumption. The case $D$ is a permutation polynomial leads to the same contradiction. This completes the proof.
\end{proof}

The size of the nucleus is an invariant under isotopism and provides a measure for the non-associativity of a commutative semifield. Lemma \ref{lemma_N} suggests that the associated semifield $S_Q$ of a planar function $Q(x)$ of the form described in the lemma behaves in two extremes: the size of its nucleus is either the maximum possible or the minimum possible. The planar functions that have associated semifields isotopic to a finite field have been characterized in \cite{coulter}, and those of the prescribed form is implicitly described in Lemma \ref{lemma_ps}. This provides some evidence to the conjecture that the planar functions of this form are all known. There is not much that we can say about the middle nucleus. In \cite{chk}, the authors have studied the equivalent forms of planar functions whose corresponding commutative semifields have specified nuclei. We deal with the planar functions of our special form in the following lemma.

\begin{lemma}\label{lemma_equiv}
Let $q$ be odd, $n\ge 2$ be an integer, and $w$ be a nonsquare in $\F_{q^n}$. Let $L(X)$ be a reduced $q$-polynomial over $\F_{q^n}$ such that \eqref{eqn_L_ker} holds. Assume that $Q(X)=L(X)^2-wX^2$ is planar over $\F_{q^n}$ and  the semifield $S_Q=(\F_{q^n},+,\circ)$ with multiplication as in \eqref{eqn_circ} has a middle nucleus of size $q^m$ with $1\le m<n$. Then $Q(X)$ is equivalent to either
\begin{enumerate}
\item[(i)]  $A(X)^2-w'X^{2q^k}$, with $\gcd(k,m)=1$, $w'$ a nonsquare and $A(x)$ a $q^m$-polynomial, or
\item[(ii)] $\left(L(1)X+T_1(\Delta)+R(\Delta)\right)^2-w\left(X+T_0(\Delta)+L(1)w^{-1}R(\Delta)\right)^2$, where $\Delta=X^{q^m}-X$, $T_0,\,T_1$ are $q^m$-polynomial and $R$ is a nonzero $q$-polynomial.
\end{enumerate}
\end{lemma}
\begin{proof}
Take the same notation as in the proof of Lemma \ref{lemma_N}, and write $s=n/m$. Let $M_1(X)=\sum_{i=0}^{s-1}f_i(X^{q^i})$ and $M_2(X)=\sum_{i=0}^{s-1}g_i(X^{q^i})$ be their $q^m$-decompositions respectively, where the $f_i$'s and $g_i$'s are reduced $q^m$-polynomials. Set $I_1=\{i:\,f_i\ne 0\}$ and $I_2=\{i:\,g_i\ne 0\}$. Both $I_1$ and $I_2$ are non-empty subsets of $\{0,1,\cdots, s-1\}$.

For $a\in\F_{q^m}$, we have $K(a)=a*_Ka=a^2$, so \eqref{eqn_M} gives that $M_1(a)^2-wM_2(a)^2=L(1)M_1(a^2)-wM_2(a^2)$. In other words,
\begin{equation}\label{eqn_M_new}
M_1(X)^2-wM_2(X)^2\equiv L(1)M_1(X^2)-wM_2(X^2)\pmod{X^{q^m}-X}.
\end{equation}
It is similar to \eqref{eqn_M12ff} and the same argument using Remark \ref{remark_size} show that either
\begin{enumerate}
\item $M_1(X)\equiv c_1X^{q^k}+c_2X^{q^l}$ and $M_2(X)\equiv c_3X^{q^k}+c_4X^{q^l}\pmod{X^{q^m}-X}$, or
\item $M_1(X)\equiv c_1'X^{q^u}$ and $M_2(X)\equiv c_2'X^{q^v}\pmod{X^{q^m}-X}$,
\end{enumerate}
where $0\le k, l,u,v\le m-1$, $k<l$, the $c_i$'s and $c_i'$'s are constants with $c_1c_2=wc_3c_4\ne 0$. Notice that $c_2'\ne 0$, since $M_2(X)$ is a permutation polynomial over $\F_{q^n}$. If $c_1'=0$, then we set $u=v$ for uniform treatment. The monomials in $M_t(X)\pmod{X^{q^m}-X}$ correspond to the components of $M_t(X)$ in its $q^m$-decomposition for $t\in\{1,2\}$, so $k,l\in I_1\cap I_2$ in the case (1), $v\in I_2$ in the case (2), and $u\in I_1$ in the case (2) if $c_1'\ne 0$.

In the case (2), after plugging them back into \eqref{eqn_M_new} and comparing coefficients we get $c_1'=L(1)$ and $c_2'=1$ if $u\ne v$ and $c_1'^2-wc_2'^2=L(1)c_1'-wc_2'$ if $u= v$. Now consider the special case $c_1'=0$. We deduce that $c_2'=1$, $M_2(a)=a^{q^u}$ for $a\in\F_{q^m}$, so $M_2$ maps $\F_{q^m}$ to $\F_{q^m}$ bijectively. On the other hand, $L(M_2(a))=M_1(a)=c_1'a^{q^u}=0$ for all $a\in\F_{q^m}$. This shows that $\F_{q^m}$ lies in the kernel of $L$. Therefore, there exists a $q$-polynomial $R$ such that $L(X)=R(X^{q^m}-X)$, so $Q(X)$ is of the second form in the lemma. We thus assume that $c_1'\ne 0$ in the case (2).

For $a\in\F_{q^m}$ and $x\in\F_{q^n}$, we plug $x+a$, $x$ and $a$ into \eqref{eqn_M} and take their linear combination to get $M_1(a)M_1(x)-wM_2(a)M_2(x)=L(1)M_1(ax)-wM_2(ax)$. Here, we have used the fact $K(x)=x*_Kx$ and $a*_Kx=ax$ for $a\in\F_{q^m}$. Since both $M_1(X)$ and $M_2(X)$ are reduced, we have $M_1(a)M_1(X)-wM_2(a)M_2(X)=L(1)M_1(aX)-wM_2(aX)$. By expanding it using the $q^m$-decompositions and comparing exponents of monomials, we get
\begin{equation}\label{eqn_figi}
(M_1(a)-L(1)a^{q^i})f_i(X^{q^i})=w(M_2(a)-a^{q^i})g_i(X^{q^i}),\quad 0\le i\le s-1.
\end{equation}
If $i$ is an integer such that $0\le i\le s-1$ and neither $M_1(X)-L(1)X^{q^i}$ nor $M_2(X)-X^{q^i}$ is zero modulo $X^{q^m}-X$, then we claim that $i$ is either in both of $I_1$ and $I_2$ or  in neither of them. There exists $a\in\F_{q^m}$ such that  $M_1(a)-L(1)a^{q^i}\ne 0$ by the assumption, so it follows from \eqref{eqn_figi} that $f_i$ is equal to  $g_i$ multiplied by a constant. Similarly, $g_i$ is equal to  $f_i$ multiplied by a constant, so the claim follows.

After these preparations, we are now ready to handle each case separately.
In the case (1), we claim that $I_1=I_2=\{k,l\}$,  $M_2(X)=g_k(X^{q^k})+g_l(X^{q^l})$ and $M_1(X)=c_1c_3^{-1}g_k(X^{q^k})+c_2c_4^{-1}g_l(X^{q^l})$.  For each $0\le i\le s-1$, neither $M_1(X)-L(1)X^{q^i}$ nor $M_2(X)-X^{q^i}$ is zero modulo $X^{q^m}-X$, so $I_1=I_2$ by the preceding claim. For $i\in I_1=I_2$, there is a nonzero constant $d_i$ such that $g_i=d_if_i$ and plugging it in \eqref{eqn_figi} we get
\begin{equation}\label{eqn_a_pf}
(L(1)-wd_i)a^{q^i}=M_1(a)-wd_iM_2(a)=(c_1-wd_ic_3)a^{q^k}+(c_2-wd_ic_4)a^{q^l}.
\end{equation}
It is straightforward to check that the right hand side has at least one nonzero coefficient, so $L(1)-wd_i\ne 0$. It follows that  $i=k$ or $i=l$. Comparing the coefficients of $a^{q^l}$ in \eqref{eqn_a_pf} in the case $i=k$ gives $d_k=w^{-1}c_4^{-1}c_2=c_3c_1^{-1}$. Similarly, we get $d_l=c_4c_2^{-1}$ in the case $i=l$. This proves the claim. We then compute that
\[
Q(M_2(x))=M_1(x)^2-wM_2(x)^2=(1-wd_k^2)g_k(x^{q^k})^2+(1-wd_l^2)g_l(x^{q^l})^2.
\]
Since $Q(M_2(x))$ is a planar function equivalent to $Q(x)$, one of $g_k$ and $g_l$ must be a permutation polynomial by Lemma \ref{lemma_L_perm}. Consider the case $g_k$ is a permutation polynomial. With $y=g_k(x^{q^k})^{q^{l-k}}$ we have $x^{q^k}=h_k(y^{q^{k-l}})$ for some $q^m$-polynomial $h_k$, and so $g_l(x^{q^l})=g_l(h_k(y^{q^{k-l}})^{q^{l-k}})=A(y)$ for some $q^m$-polynomial $A$. We now have $Q(M_2(x))=(1-wd_k^2)y^{2q^{k-l}}+(1-wd_l^2)A(y)^2$ with $A$ a $q^m$-polynomial, and so $Q(x)$ is equivalent to one of the first form in the lemma. Since the nucleus of $S_Q$ is $\F_q$ by Lemma \ref{lemma_N}, we have $\gcd(k-l,m)=1$.  The case $g_l$ is a permutation polynomial is dealt with similarly.

In the case (2) with $u\ne v$,  we claim that $I_1=\{u\}$ and $I_2=\{v\}$. Recall that $c_1'=L(1)\ne 0$, $c_2'=1$, $u\in I_1$ and $v\in I_2$ in this case. For each $i\in\{0,1,\cdots,s-1\}\setminus\{u,v\}$, neither $M_1(X)-L(1)X^{q^i}$ nor $M_2(X)-X^{q^i}$ is zero modulo $X^{q^m}-X$ in this case, so $I_1\setminus\{u,v\}=I_2\setminus\{u,v\}$. If $i\in I_1\setminus\{u,v\}=I_2\setminus\{u,v\}$, then there exists a nonzero constant $d_i$ such that $g_i=d_if_i$, and \eqref{eqn_figi} reduces to $(L(1)-wd_i)a^{q^i}=c_1'a^{q^u}-wd_ic_2'a^{q^v}$. This is impossible since $c_1'\ne 0$ and  $wd_ic_2'\ne 0$. Hence $I_1,I_2$ are both subsets of $\{u,v\}$. By setting $i=v$ in \eqref{eqn_figi}, we get $f_v=0$. Similarly, $g_u=0$. This proves the claim. By the same argument in the previous case, $Q(x)$ is equivalent to one of the first form in the lemma.

Finally, consider the case (2) with $u=v$ and $c_1'\ne 0$. In this case, $u$ is in both $I_1$ and $I_2$. As before we have $I_1\setminus\{u\}=I_2\setminus\{u\}$. If $I_1=I_2=\{u\}$, then from the fact $M_2(X)$ is a permutation polynomial and $M_1(X)=L(M_2(X))\pmod{X^{q^n}-X}$ we deduce that $L$ is a $q^m$-polynomial, contradicting \eqref{eqn_L_ker}.
Hence $I_1\setminus\{u\}=I_2\setminus\{u\}\ne\emptyset$. For $t\in I_1\setminus\{u\}$, there exists a constant $d_t\ne 0$ such that $g_t=d_tf_t$ and \eqref{eqn_figi} reduces to $c_1'a^{q^u}-L(1)a^{q^t}=wd_t(c_2'a^{q^u}-a^{q^t})$ for $a\in\F_{q^m}$. By comparing coefficients we get $L(1)=wd_t$ and $c_1'=wd_tc_2'$. Together with $c_1'^2-wc_2'^2=L(1)c_1'-wc_2'$, we deduce that $c_2'=1$, $c_1'=L(1)$ and $d_t=L(1)w^{-1}$. To sum up, we have $M_2(X)-g_u(X^{q^u})=L(1)w^{-1}(M_1(X)-f_u(X^{q^u}))$, $M_1(X)\equiv L(1)X^{q^u}$ and $M_2(X)\equiv X^{q^u}\pmod{X^{q^m}-X}$. The $\F_{q^m}$-linear maps $x\mapsto f_u(x)-L(1)x$ and $x\mapsto g_u(x)-x$ both have $\F_{q^m}$ in their kernels, so there exist $q^m$-polynomials $T_1,T_0$ such that $T_1(X^{q^m}-X)=f_u(X)-L(1)X$ and $T_0(X^{q^m}-X)=g_u(X)-X$. Meanwhile, $M_1(X)-f_u(X^{q^u})$ has $\F_{q^m}$ in the kernel, so there exists a $q$-polynomial $R$ such that it is equal to $R(X^{q^{m+u}}-X^{q^u})$. We thus have $M_1(X)=L(1)X^{q^u}+T_1(X^{q^m}-X)+R(X^{q^{m+u}}-X^{q^u})$, $M_2(X)=X^{q^u}+T_0(X^{q^m}-X)+L(1)w^{-1}R(X^{q^{m+u}}-X^{q^u})$, and $Q(X)$ is equivalent to the second form in this case.
\end{proof}

For the rest of this section, we will only consider the simplest cases as a demonstration of techniques, where the associated semifields are rank two commutative semifields. They correspond to Case (i) in Lemma \ref{lemma_equiv} with $n=2m$. We start with a technical lemma.
\begin{lemma}\label{lemma_Psi} Let $q$ be an odd prime power, $m$ be a positive integer, and take $\zeta\in\F_{q^{2m}}$ such that $\zeta^{q^m-1}=-1$.  Let $\Psi:\,\F_{q^{2m}}^2\mapsto\F_{q^m}^3$ be the map defined by
\begin{equation}\label{eqn_Psi}
\Psi(x_0\zeta+x_1,y_0\zeta+y_1):=(x_1y_1,\,x_0y_0,\,x_0y_1+x_1y_0),\quad\forall\, x_0,x_1,y_0,y_1\in\F_{q^m}.
\end{equation}
Then its image set is equal to $\{(A,B,C)\in\F_{q^m}^3:\,C^2-4AB \text{ is a square in $\F_{q^m}$}\}$, and $\Psi(x,y)=(0,0,0)$ if and only if  $x=0$ or $y=0$.
\end{lemma}
\begin{proof}
If $(A,B,C)=\Psi(x,y)$ with $x=x_0\zeta+x_1$, $y=y_0\zeta+y_1$ ($x_i,y_i\in\F_{q^m}$), then $A=x_1y_1$, $B=x_0y_0$, $C=x_0y_1+x_1y_0$, and $C^2-4AB=(x_0y_1-x_1y_0)^2$ is a square in $\F_{q^m}$. If further $A=B=C=0$, then it is straightforward to show that at least one of $x$ and $y$ is $0$.

Conversely, suppose that $(A,B,C)\in\F_{q^m}^3$ satisfies that $C^2-4AB$ is a square. We can directly check that $\Psi(\zeta,B\zeta+C)=(A,B,C)$ if $A=0$ and $\Psi(1,C\zeta+A)=(A,B,C)$ if $B=0$.  Now assume that $AB\ne 0$, and let $t$ be a solution to $BX^2-CX+A=0$. It is now routine to check that $t\in\F_q^*$ and $\Psi(\zeta+t,B\zeta+At^{-1})=(A,B,C)$.
\end{proof}

\begin{lemma}\label{lemma_specialQ}
Suppose that $m\ge 3$, $1\le k\le m-1$ and $\gcd(k,m)=1$. Let $q$ be odd and $w$ be a nonsquare in $\F_{q^{2m}}$. Then $Q(X)=(X^{q^m}-X)^2-wX^{2q^k}$ is not planar over $\F_{q^{2m}}$.
\end{lemma}
\begin{proof}
Take $\zeta\in\F_{q^{2m}}$ such that $\zeta^{q^m-1}=-1$, and let $\Psi$ be as defined in \eqref{eqn_Psi}. Write $x=x_0\zeta+x_1$, $y=y_0\zeta+y_1$ with $x_i,\,y_i\in\F_{q^m}$, and set $(A,B,C)=\Psi(x,y)$. Then $Q(x+y)=Q(x)+Q(y)$ if and only if
\begin{equation}\label{eqn_ABC0}
-wA^{q^k}+(4\zeta^2 B-w\zeta^{2q^k}B^{q^k})-w\zeta^{q^k} C^{q^k}=0.
\end{equation}
The left hand side is equal to $\frac{1}{2}(Q(x+y)-Q(x)-Q(y))$. By Lemma \ref{lemma_Psi}, $Q(x)$ is planar if and only if there is no triple $(A,B,C)\in\F_{q^m}^3\setminus\{(0,0,0)\}$ such  that  $C^2-4AB$ is a square and \eqref{eqn_ABC0} holds. To show that $Q(x)$ is not  planar, we need to establish the existence of such a triple.

By raising both sides of \eqref{eqn_ABC0} to the $q^m$-th power we get another equation, and together with \eqref{eqn_ABC0} we deduce that
\begin{equation}\label{eqn_AC}
A^{q^k}=2( w^{-q^m}+w^{-1})\zeta^2B-\zeta^{2q^k}B^{q^k},\;
C^{q^k}=2\zeta^{2-q^k}(w^{-1}- w^{-q^m})B,
\end{equation}
We then compute that
\begin{align*}
(C^2&-4AB)^{q^k}=4\zeta^{4-2q^k}(w^{-1}- w^{-q^m})^2B^2-4B^{q^k}\left(2( w^{-q^m}+w^{-1})\zeta^2B-\zeta^{2q^k}B^{q^k}\right)\\
&=4\zeta^{4-2q^k}B^2\left((w^{-1}- w^{-q^m})^2-2( w^{-q^m}+w^{-1})\zeta^{2(q^k-1)}B^{q^k-1}+\zeta^{4(q^k-1)}B^{2(q^k-1)}\right)\\
&=4\zeta^{4-2q^k}B^2\left((w^{-1}+w^{-q^m}-\zeta^{2(q^k-1)}
B^{q^k-1})^2-4w^{-1-q^m}\right).
\end{align*}
Since $\zeta^2$ is a nonsquare in $\F_{q^m}^*$ and $\gcd(q^k-1,q^m-1)=q-1$, we need to find $z\in\F_{q^m}^*$ such that $H(z):=(w^{-1}+w^{-q^m}-\zeta^{2(q^k-1)}z^{q-1})^2-4w^{-1-q^m}$ is a nonsquare. Then by taking $B\in\F_{q^m}$ such that $B^{q^k-1}=z^{q-1}$ and setting $A,\,C$ as in \eqref{eqn_AC}, we get the desired triple. It remains to establish the existence of such an element.

Let $\rho$ be the multiplicative character of order two of $\F_{q^m}^*$, i.e., $\rho(x)=1$ if $x$ is a nonzero square and $\rho(x)=-1$ otherwise. We extend it to $\F_{q^m}$ by setting $\rho(0)=0$. Let $\overline{\F}_{q^m}$ be the algebraic closure of $\F_{q^m}$. If $H(X)=h(X)^2$ for some polynomial $h(X)\in\overline{\F}_{q^m}[X]$, then
\[
(w^{-1}+w^{-q^m}-\zeta^{2(q^k-1)}X^{q-1}+h(X))(w^{-1}+w^{-q^m}-\zeta^{2(q^k-1)}X^{q-1}-h(X))=4w^{-1-q^m}.
\]
It implies that the two factors on the left hand side both have degree $0$, which is impossible. We can now apply \cite[Theorem 6.2.2]{handbookff} to see that
$|\sum_{z\in\F_{q^m}}\rho(H(z))|\le (2(q-1)-1)q^{m/2}$.
It is straightforward to check that $(2(q-1)-1)q^{m/2}<q^m-1$  holds for all $q$ odd and $m\ge 3$, so there exists $z\in\F_{q^m}^*$ such that $H(z)$ is a nonsquare. This completes the proof.
\end{proof}

\begin{thm}\label{thm_Q}
Let $k,m$ be positive integers with $m\ge 3$, $\gcd(k,m)=1$, and let $q=p^e$ with $p$ an odd prime.  Suppose that either of the following holds: (1) $q\ge 4n^2-8n+2$, (2) $p>2(em)^2-(4-2\sqrt{3})em+(3-2\sqrt{3})$, (3) $m=3$. Then $Q(X)=(aX+bX^{q^m})^2-wX^{2q^k}$ is planar over $\F_{q^{2m}}$ if and only if $ab=0$, where $a,b\in\F_{q^{2m}}$ and $w$ is a nonsquare in $\F_{q^{2m}}^*$.
\end{thm}
\begin{proof}
Take $\zeta\in\F_{q^{2m}}$ such that $\zeta^{q^m-1}=-1$, and let $\Psi$ be as defined in \eqref{eqn_Psi}. Set $s:=a+b$, $\beta:=(a-b)\zeta$. The case $s=0$ has been handled in Lemma \ref{lemma_specialQ}, so assume that $s\ne 0$. By multiplying $Q(x)$ by a nonzero constant if necessary, we set $s=1$ without loss of generality. If $\beta$ is in $\F_{q^m}$, then $a=\frac{1+\beta\zeta^{-1}}{2}$, $b=\frac{1-\beta\zeta^{-1}}{2}$, $-ab^{-1}=\frac{\beta+\zeta}{\beta-\zeta}=(\beta-\zeta)^{q^m-1}$, and $a^{-2}(\beta-\zeta)^{-2}Q\left((\beta-\zeta) X\right)$ is of the form in Lemma \ref{lemma_specialQ}. Therefore, we only need to consider the case $s=1$ and $\beta\not\in\F_{q^m}$. As in the case of Lemma \ref{lemma_specialQ}, $Q(x)$ is planar if and only if there is no triple $(A,B,C)\in\F_{q^m}^3\setminus\{(0,0,0)\}$ such  that  $C^2-4AB$ is a square and
\begin{equation}\label{eqn_ABC1}
 A-wA^{q^k}+\beta^2 B-w\zeta^{2q^k}B^{q^k}+\beta C-w\zeta^{q^k} C^{q^k}=0.
\end{equation}
The left hand side equals $\frac{1}{2}(Q(x+y)-Q(x)-Q(y))$ in the case $(A,B,C)=\Psi(x,y)$.

Assume that $Q(x)$ is planar. We claim that the $\F_q$-linear map
\[
\Upsilon:\,\F_{q^m}^2\mapsto\F_{q^{2m}},\,(A,C)\rightarrow  A-wA^{q^k}+ \beta C-w\zeta^{q^k}C^{q^k}
 \]
is a bijection. Otherwise, there exists a pair $(A,C)\ne (0,0)$ such that \eqref{eqn_ABC1} holds with $B=0$. Since $C^2-4\cdot A\cdot 0=C^2$ is trivially a square, this is impossible by our assumption. Therefore, for each $B=u\in\F_{q^m}$, there is a unique pair $(A,C)=(f(u),g(u))$ such that $\Upsilon(f(u),g(u))=-\beta^2 u+w\zeta^{2q^k}u^{q^k}$, where $f(X),g(X)\in\F_{q^m}[X]$.  It is clear that both $f$ and $g$ are $\F_q$-linear. By the definition of $f$ and $g$ we have
\begin{align}\label{eqn_ABC}
f(u)-wf(u)^{q^k}+\beta^2 u-w\zeta^{2q^k}u^{q^k}+\beta g(u)-w\zeta^{q^k} g(u)^{q^k}=0,\quad\forall\, u\in\F_{q^m}.
\end{align}
If $f(u)=0$ for some $u\ne 0$, then $(A,B,C)=(0,u,g(u))$ satisfies \eqref{eqn_ABC1} and $C^2-4AB=g(u)^2$ is a square. This contradicts our assumption, so the map $u\mapsto f(u)$ is a bijection.

By the planarity of $Q(x)$, $g(u)^2-4uf(u)$ is a nonsquare for all $u\in\F_{q^m}^*$, so  $S(-g,-f)=(\F_{q^{2m}},+,\circ)$ as defined in Theorem \ref{thm_rtcs_def} with $t=\zeta$ is a RTCS. Recall that for $x,y\in\F_{q^{2m}}$ and $(A,B,C)=\Psi(x,y)$, we have
\[
x\circ y=(-g(B)+C)\zeta-f(B)+A.
\]
Let $M$ be the the $\F_q$-linear map such that $M(z_0\zeta+z_1):=\Upsilon(z_1,z_0)$ for $z_0,z_1\in\F_{q^m}$. It is nondegenerate, and with $(A,B,C)=\Psi(x,y)$ we have
\begin{align*}
M(x\circ y)&=\Upsilon(-f(B)+A,-g(B)+C)=-\Upsilon(f(B),g(B))+\Upsilon(A,C)\\
&=\beta^2 B-w\zeta^{2q^k}B^{q^k}+\Upsilon(A,C)=\frac{1}{2}(Q(x+y)-Q(x)-Q(y)).
\end{align*}
Therefore, the semifield $S_Q$ defined by $Q$ is isotopic to $S(-g,-f)$.

Under the conditions in the theorem, $S(-g,-f)$ is isotopic to either a finite field or a Dickson semifield by Theorem \ref{thm_rtcs}. By \cite[Example 2]{CG}, $S(-g,-f)$ is isotopic to a finite field if and only if $g(u)=cu$ and $f(u)=du$, where $c^2-4d$ is a nonsquare. In this case, \eqref{eqn_ABC} becomes
$(d+\beta c+\beta^2)u=w(d +\zeta  c+\zeta^{2 })^{q^k}u^{q^k}.$ It holds for all $u$, so $d +\zeta  c+\zeta^{2 }=0$ and $d+\beta c+\beta^2=0$. Since $\zeta^2\in\F_{q^m}$, we deduce that $c=0$, $d=-\zeta^2$ and $\beta=\pm\zeta$. It follows from $a+b=1$ and $a-b=\beta\zeta^{-1}$ that $(a,b)=(1,0)$ or $(0,1)$.

From now on, assume that $S(-g,-f)$ is not isotopic to a finite field but is isotopic to a Dickson semifield $S(0,\lambda x^{\theta})=(\F_{q^{2m}},+,\circ')$ as defined in Theorem \ref{thm_rtcs_def} with $t=\zeta$, where $\lambda$ is a nonsquare in $\F_{q^m}$ and $\theta\in\Aut(\F_{q^{2m}})$. By Lemma \ref{lemma_N}, the semifield $S_Q$ and so $S(-g,-f)$ has nucleus $\F_{q}$. It follows that $x^\theta=x^{q^\ell}$ for some $1\le\ell\le m-1$ such that $\gcd(\ell,m)=1$, cf. \cite[Theorem 10.16]{handbook}. By \cite[Theorem 2.6]{coulter}, there exist nondegenerate linear maps $L,N:\,\F_{q^{2m}}\mapsto\F_{q^{2m}}$ and $\alpha\in\F_{q^m}^*$ such that
\begin{equation}\label{eqn_iso}
L(x\circ y)=N(x)\circ' (\alpha N(y)).
\end{equation}
There exist linear maps $L_i:\,\F_{q^m}\mapsto\F_{q^m}$, $1\le i\le 4$, such that $N(x)=(L_1(x_0)+L_2(x_1))\zeta+L_3(x_0)+L_4(x_1)$ for $x=x_0\zeta+x_1$. Write $L_1(X)=\sum_{i=0}^{em-1}a_iX^{p^i}$ and $L_3(X)=\sum_{i=0}^{em-1}b_iX^{p^i}$, with $a_i,b_i\in\F_{q^m}$. We consider \eqref{eqn_iso} in three cases.
\begin{enumerate}
\item In the case $x_1=y_1=0$, it shows that $L(-g(x_0y_0)\zeta-f(x_0y_0))$ is equal to
\begin{equation*}
\lambda\alpha^{q^\ell}L_1(x_0)^{q^\ell}L_1(y_0)^{q^\ell}+\alpha L_3(x_0)L_3(y_0)+\alpha(L_1(x_0)L_3(y_0)
+L_3(x_0)L_1(y_0))\zeta.
\end{equation*}
The coordinates with respect to the basis $\{1,\zeta\}$ in the above expression are polynomials  in $x_0y_0$. By comparing coefficients we get $\lambda\alpha^{q^\ell}a_i^{q^\ell}a_j^{q^\ell}+\alpha b_{i+el}b_{j+el}=0$, and $a_{i+el}b_{j+el}+a_{j+el}b_{i+el}=0$ for $i\ne j$. The subscripts are read modulo $em$ here. We observe that   $-a_ia_jb_{i+el}b_{j+el}$ is a nonsquare and $-a_{i+el}a_{j+el}b_{i+el}b_{j+el}$ is a square, yielding that $a_ia_{i+el}a_ja_{j+el}$ is a nonsquare in the case $i\ne j$ and the terms involved are nonzero. Assume that $a_ua_v\ne 0$ for $0\le u<v\le em-1$. We claim that none of the elements in $\{a_{u+rel},\,b_{v+sel}:\,0\le r,s\le m-1\}$ is zero. It follows from the first equation with $(i,j)=(u,v)$ that $b_{u+el}b_{v+el}\ne 0$, and follows  from the second equation with $(i,j)=(u-e\ell,u)$ and $(i,j)=(v-e\ell,v)$ that none of $a_{u+el},a_{v+el},b_u,b_v$ is zero. The claim then follows by induction. For $0\le i<j\le m-1$, exactly one of $a_{u+iel}a_{u+(i+1)el}$ and $a_{u+jel}a_{u+(j+1)el}$ is a square and the other is a nonsquare by our previous observation. This is impossible when $m\ge 3$. We conclude that $L_1$ is a monomial, and similarly $L_3$ is a monomial. Write $L_1(x_0)=c_1x_0^{p^{i}}$ and $L_3(x_0)=c_3x_0^{p^{i}}$ for some $i$ and constants $c_1,c_3$. Since $N$ is nondegenerate, $(c_1,c_3)\ne(0,0)$. In this case, \eqref{eqn_iso} reduces to
\begin{equation}\label{eqn_x0y0}
-L(g(u)\zeta+f(u))=2\alpha c_1c_3 u^{p^i}\zeta+\alpha c_3^2u^{p^i}+\lambda\alpha^{q^\ell}c_1^{2q^\ell}u^{q^\ell p^i},\;\forall\, u\in\F_{q^m}.
\end{equation}
\item In the case $x_1=y_0=0$, \eqref{eqn_iso} gives that $L(x_0y_1\zeta)$ is equal to
\begin{align*}
\alpha x_0^{p^i}(c_1L_4(y_1)  +c_3L_2(y_1))\zeta+\alpha c_3x_0^{p^i}L_4(y_1)+\lambda\alpha^{q^\ell}c_1^{q^\ell}x_0^{p^iq^\ell}L_2(y_1)^{q^\ell}.
\end{align*}
As in the previous case, we know that $\alpha c_3x_0^{p^i}L_4(y_1)+\lambda\alpha^{q^\ell}c_1^{q^\ell}x_0^{p^iq^\ell}L_2(y_1)^{q^\ell}$ is a polynomial in $x_0y_1$. This is possible if and only if there exists constants $c_2,\,c_4$ such that $L_2(x)=c_2x^{p^i}$ and $L_4(x)=c_4x^{p^i}$ for the same $i$. The nondegeneracy of $N$ requires that $c_1c_4-c_2c_3\ne 0$. In this case, \eqref{eqn_iso} takes the form
\begin{equation}\label{eqn_x1y0}
L(u\zeta)=\alpha(c_1c_4+c_2c_3)u^{p^i}\zeta+\alpha c_3c_4u^{p^i}+\lambda \alpha^{q^\ell} c_1^{q^\ell}c_2^{q^\ell}u^{q^\ell p^i},\;\forall\,u\in\F_{q^m}.
\end{equation}
\item In the case $x_0=y_0=0$, \eqref{eqn_iso} gives that
\begin{equation}\label{eqn_x1y1}
L(u)=2\alpha c_2c_4u^{p^i}\zeta+\alpha c_4^2u^{p^i}
+\lambda\alpha^{q^\ell}c_2^{2q^\ell}u^{p^iq^\ell},\;\forall\,u\in\F_{q^m}.
\end{equation}
\end{enumerate}
We compute $L(g(u)\zeta)+L(f(u))$ using \eqref{eqn_x1y0} and \eqref{eqn_x1y1} and add it to
\eqref{eqn_x0y0} to cancel out the left hand side. The coordinate of $\zeta$ on the right hand side gives that
\begin{equation}\label{eqn_c}
2c_2c_4f(u)^{p^i}+(c_1c_4+c_2c_3)g(u)^{p^i}+2c_1c_3u^{p^i}=0.
\end{equation}

We claim that none of $c_1$, $c_2$, $c_3$, $c_4$ is zero. If $c_2c_4=0$, then $c_1c_4+c_2c_3\ne 0$ and \eqref{eqn_c} gives that $g(u)=c_5u$ for some constant $c_5$. The equation \eqref{eqn_ABC} now takes the form
\[
f(u)-wf(u)^{q^k}=-(\beta^2+\beta c_5) u+(w\zeta^{2q^k}+w\zeta^{q^k} c_5^{q^k})u^{q^k}.
\]
We obtain another equation by raising both sides to the $q^m$-th power, and then deduce that both $f(u)$ and $f(u)^{q^k}$ are linear combinations of $u$ and $u^{q^k}$. Since $\gcd(m,k)=1$, this is possible only if $f$ has degree $1$. However, $S(-g,-f)$ is then isotopic to a finite field: a contradiction. Hence $c_2c_4\ne 0$. If $c_1c_3=0$, then with the role of $z=f(u)$ and $u$ interchanged and $g$ considered as a function of $z$, we derive the same  contradiction. This proves the claim.

From \eqref{eqn_c}, we see that  $f(u)=d_1g(u)+d_2u$ for some constants $d_1,\,d_2\in\F_{q^m}$ with $d_2\ne 0$. Canceling $f$ from \eqref{eqn_ABC} by substitution, we get
\begin{equation}\label{eqn_gonly}
(\beta+d_1)g(u)-w(d_1+\zeta)^{q^k}g(u)^{q^k}=-(d_2+\beta^2)u+w(d_2+\zeta^{2})^{q^k}u^{q^k}.
\end{equation}
We claim that $\frac{w(d_1+\zeta)^{q^k}}{\beta+d_1}\in \F_{q^m}^*$. Otherwise, raising both sides of \eqref{eqn_gonly} to the $q^m$-th power, we get another equation that is linear independent with \eqref{eqn_gonly}. We then deduce that both $g(u)$ and $g(u)^{q^k}$ are linear combinations of $u,u^{q^k}$. This is possible only if $g(u)=c_6u$ for a constant $c_6$, but then $f(u)$ has degree $1$: a contradiction. This proves the claim.

After dividing both sides of \eqref{eqn_gonly} by $\beta+d_1$, the left hand side is in $\F_{q^m}[u]$, so should be the right hand side. This gives that both $\frac{d_2+\beta^2}{d_1+\beta}$ and $ \frac{d_2+\zeta^{2}}{d_1+\zeta} $ are in $\F_{q^m}$. Since $d_2+\zeta^{2}$ is in $\F_{q^m}$ but $d_1+\zeta$ is not, we must have $d_2=-\zeta^2$. We have $d_2+\beta^2\ne 0$, since otherwise $\beta^2=\zeta^2$ and $S_Q$ is isotopic to a finite field as we have shown. Now set $z:=g(u)$. The equation \eqref{eqn_gonly} gives that $u=h_1z+h_2z^{q^k}$ for some constants $h_1,h_2\in\F_{q^m}^*$. In particular, this shows that $u\mapsto z=g(u)$ is a bijection. It follows that $f(u)=h_3z-\zeta^2h_2z^{q^k}$ with $h_3=d_1-\zeta^2 h_1$, and we have
\begin{align*}
g(u)^2-4uf(u)&=z^2-(h_1z+h_2z^{q^k})(h_3z-\zeta^2h_2z^{q^k})\\
&=z^2(h_5+h_6z^{q^k-1}+\zeta^2 h_2^2z^{2(q^k-1)}),
\end{align*}
where $h_5=1-h_1h_3$, $h_6=\zeta^2h_1h_2-h_2h_3$.

Set $H(X):=h_5+h_6X^{q-1}+\zeta^2 h_2^2X^{2(q-1)}$. We have $H(X)=\left(\zeta h_2X^{q-1}+\frac{h_6}{2\zeta h_2}\right)^2+\frac{4\zeta^2h_2^2h_5-h_6^2}{4\zeta^2h_2^2}$.
We directly compute that $4\zeta^2 h_2^2h_5-h_6^2=(4\zeta^2-d_1^2)h_2^2\ne 0$, so $H(X)$ is not a square in $\overline{\F}_{q^m}[X]$ by the same argument in the proof of Lemma \ref{lemma_specialQ}, where $\overline{\F}_{q^m}$ is the algebraic closure of $\F_{q^m}$. The  same exponential sum bound there  establishes the existence of $z$ such that $h_5+h_6z^{q^k-1}+\zeta^2 h_2^2z^{2(q^k-1)}$ is a square. If we take $u$  such that $z=g(u)$, then $g(u)^2-4uf(u)$ is a nonsquare. This contradiction completes the proof.
\end{proof}

In the proof of Theorem \ref{thm_Q}, we do not directly consider the planarity of $Q$. Instead, we show that the semifield $S_Q$ is a RTCS and make use of the classification results of such semifields obtained in \cite{rtcs_char1,rtcs_char2}. This is in the same spirit as in Section \ref{sect_function}, where we make use of Menichetti's classification of generalized twisted fields.  Coulter and Henderson have used this approach to characterize planar functions of certain form over $\F_{q^3}$, cf. \cite{ch_planar}. The list of known commutative semifields is short \cite{helle_sur,fs_cur,cs_zp}, and Minami and Nakagawa \cite{naka} have determined the polynomial forms of certain commutative semifields. It may be of some interest to examine the known commutative semifields one by one to check whether their isotopes will yield planar functions of the form $L(X)^2-wX^2$, but the answer is most probably negative. We need new techniques to prove or disprove the planarity of the functions of interest.

\section{Low-dimensional $\cC$-planes of even order}

In this section, we consider type $\cC$ spreads of even order. Let $q$ be even. Recall that $\beta$ is an element of order $(q^n+1)(q-1)$ in $\F_{q^{2n}}$, and $\Theta(\beta)\in\Gamma L(1,q^{2n})$ is defined by $\Theta(\beta)(x)=\beta x$.  A type $\cC$ spread $\cS$ of order $q^n$ with kernel $\F_q$ is isomorphic to the orbit of $W=\{L(x)+\delta x:\,x\in\F_{q^n}\}$ under the group $\la\Theta(\beta)\ra$, where $L(X)$ is a monic reduced $q$-polynomial and $\delta\in\F_{q^{2n}}\setminus\F_{q^n}$.
By Lemma \ref{lemma_key}, $Q(X)=(L(X)+\delta X)(L(X)+\delta^{q^n} X)\in\F_{q^n}[X]$ is a permutation polynomial of $\F_{q^n}$. For each $y\in \F_{q^{n}}$, we define the quadratic form $Q_y(x):=\tr_{\F_{q^n}/\F_q}\left((\delta+\delta^{q^n})^{-1}yQ(x)\right)$, i.e.,
\begin{equation}\label{eqn_Qy}
Q_y(x):=\tr_{\F_{q^n}/\F_q}\left((\delta+\delta^{q^n})^{-1}yL(x)^2+ (\delta^{-1}+\delta^{-q^n})^{-1}y
 x^2+yxL(x)\right).
\end{equation}
By Lemma \ref{lemma_perm_rank}, $Q_y$ has odd rank for any $y\ne 0$. Its associated bilinear form is
\[
B_y(u,v)=\tr_{\F_{q^n}/\F_q}(yuL(v)+yvL(u))=\tr_{\F_{q^n}/\F_q}\left((\tilde{L}(yu)+yL(u))v\right),
\]
where $\tilde{L}$ is the adjoint polynomial of $L$. The radical
$\rad(Q_y)=\{u\in \F_{q^n}:\,\tilde{L}(yu)+yL(u)=0\}$. By Theorem \ref{thm_N0}, $Q_y(\rad(Q_y))\ne \{0\}$ for each $y\ne 0$. We will fix these notation throughout this section. We first examine two special cases in the following examples.
\begin{example}\label{ex1}
Assume that $n$ is even, and $L(x)=x^{q^k}$ with $\gcd(n,k)=1$. In this case, the adjoint polynomial $\tilde{L}(X)=X^{q^{n-k}}$. Take $y$ to be a primitive element. Since $\gcd(q^{n-k}-q^k,q^n-1)=q^2-1$ and $\gcd(1-q^{n-k},q^n-1)=q-1$,  we have
$\rad(Q_y)=\{u\in \F_{q^n}:\,u^{q^{n-k}-q^k}=y^{1-q^{n-k}}\}=\{0\}$.
Hence, $Q_y$ has rank $n$ and $Q(X)$ is not a permutation polynomial  of $\F_{q^n}$.
\end{example}

\begin{example}\label{ex2}
Assume that $n$ is even, and $L(x)=\tr_{\F_{q^n}/\F_q}(x)$. In this case, $\tilde{L}(X)=L(X)$. Set $\Delta:=(\delta^{-1}+\delta^{-q^n})^{-1}$.  Take $y\in\F_{q^2}$ such that $y+y^q=\Delta$ if $\Delta\in\F_q^*$, and take $y=\Delta^{-1}$ otherwise. Then $y\not\in\F_q$ and $\rad(Q_y)=\{u:\,\tr_{\F_{q^n}/\F_q}(yu)=\tr_{\F_{q^n}/\F_q}(u)=0\}$.
For $u\in\rad(Q_y)$, we have
$Q_y(u)=\tr_{\F_{q^n}/\F_q}\left(\Delta y\cdot u^2\right)$. It is clear that $ \Delta y \in \F_q+\F_q\cdot y^2$ for the chosen $y$, so $Q_y$ is constantly zero on $\rad(Q_y)$. It follows that $Q(X)$ is not a permutation polynomial of $\F_{q^n}$.
\end{example}

The main objective of this section is to characterize the case $n=3$ and $n=4$ completely. The case $n=2$ can be reduced to either of the above examples with a proper choice of $\delta$. In the case $n=3$, there is the construction by Kantor in \cite{kantor_even}. We start with the case $n=4$. The strategy is to show that there is a trivial $\rad(Q_y)$ except the cases where it can be reduced to the second example above.

\begin{thm}\label{thm_even_4dim}
If $q$ is even, then there is no type $\cC$ spread  with ambient space $(\F_{q^8},+)$ and kernel $\F_q$.
\end{thm}
\begin{proof} We continue with the arguments in the beginning of this section.
We only deal with the case $\deg(L)=q^3$ and the other cases can be handled  similarly. If $L(X)=X^{q^3}+ax^{q^2}+bx^q+cx$, then by replacing $\delta$ with $\delta+c$ we assume that $c=0$. In this case, $\tilde{L}(X)=b^{q^3}X^{q^3}+a^{q^2}X^{q^2}+X^q$ and $
\tilde{L}(yu)+L(u)y=(b^{q^3}y^{q^3}+y)u^{q^3}+(a^{q^2}y^{q^2}+ay)u^{q^2}+(y^q+by)u^q$.
The associated matrix of this $q$-polynomial as in \eqref{eqn_mat} has determinant $f(y)^2$, where
\[
f(y)=y^{1+q}c_1+y^{1+q^2}c_2+y^{q+q^2}c_1^q+y^{q^3+1}c_1^{q^3}+y^{q+q^3}c_2^q+y^{q^2+q^3}c_1^{q^2},
\]
with $c_1=b^q+a^{1+q}$, $c_2=1+b^{1+q^2}$.
\begin{enumerate}
\item If at least one of $c_1$, $c_2$ is not zero, then $f(Y)$ is a nonzero polynomial with degree less than $q^4$.  For $y\in \F_{q^4}$ such that $f(y)\ne 0$,  $\rad(Q_y)=\{0\}$ and $Q_y$ has rank $4$.
\item If $c_1=c_2=0$, then $b^{1+q^2}=1$ and $b^q=a^{1+q}$. It follows that $a^{(1+q)(1+q^2)}=1$, i.e., $a$ is a $(q-1)$-st power in $\F_{q^4}^*$. Since $\gcd(q^3-1,q^4-1)=q-1$, there exists $u\in \F_{q^4}^*$ such that $a=u^{q^3-1}$. Then  $b=a^{(1+q)q^2}=u^{q^2-1}$, and so $L(x)=u^{-1}\tr_{\F_{q^4}/\F_q}(u^qx)+u^{q-1}x$. With $x'=u^qx$ and $\delta'=1+\delta u^{1-q}$, we have
    \[
    L(x)+\delta x=u^{-1}(\tr_{\F_{q^4}/\F_q}(x')+\delta' x')
    \]
    This reduces to the case in Example \ref{ex2}.
\end{enumerate}
In either case,  $Q(X)$ is not a permutation polynomial of $\F_{q^4}$. This completes the proof.
\end{proof}

The rest of this section is devoted to the classification of the case $n=3$.
\begin{lemma}\label{lemma_even_3dim}
For  $\delta\in\F_{q^6}\setminus \F_{q^3}$,
the map $x\mapsto Q(x)=(\tr_{\F_{q^3}/\F_q}(x)+\delta x)^{1+q^3}$ is a permutation of $\F_{q^3}$ if and only if $\delta^{-1}+\delta^{-q^3}\in\F_q^*$.
\end{lemma}

\begin{proof}
Set $r:=1+\delta^{-1}$, $h:=(\delta^{-1}+\delta^{-q^3})^{-1}$, and set $\tr:= \tr_{\F_{q^3}/\F_q}$ throughout this proof.
By Lemma \ref{lemma_perm_rank} and Theorem \ref{thm_N0}, $Q(X)$ is a permutation polynomial of $\F_{q^3}$ if and only if $Q_y$ is not constantly zero on $\rad(Q_y)$ for each $y\ne 0$, where $Q_y$ is as defined in \eqref{eqn_Qy}. In this case,
\[
Q_y(x)=\tr((\delta+\delta^{q^3})^{-1}y)\cdot\tr(x^2)
+\tr(hyx^2)+\tr(yx)\cdot\tr(x).
\]
If $y\in \F_q^*$, then $\rad(Q_y)=\F_{q^3}$ and $Q_y(x)=y\tr(cx^2)$, with $c=h+1+\tr((\delta+\delta^{q^3})^{-1})$. If $y\not\in \F_q$, then $\rad(Q_y)=\{u:\,\tr(u)=\tr(yu)=0\}=\F_q\cdot (y^q+y^{q^2})$, and
\[
  Q_y(y^q+y^{q^2})=y^{1+q+q^2}\cdot\tr\left(h(y^{q-q^2}+y^{q^2-q})\right).
\]
If $h\not\in\F_q^*$, then $ Q_h(h^q+h^{q^2})=0$ and $Q_h(\rad(Q_h))=\{0\}$. Hence we must have $h\in \F_q^*$ in order for $Q$ to be a permutation polynomial, and this proves the necessity part.

Now assume that  $h\in \F_q^*$. In this case, it is straightforward to show that $c=\tr(r^{1+q^3})h$ and the minimal polynomial of $rh$ over $\F_{q^3}$ is $X^2+X+r^{1+q^3}h^2=0$. By \cite[Theorem 2.25]{ff}, $\tr_{\F_{q^3}/\F_2}(h^2r^{1+q^3})=1$, which in particular implies that $\tr(h^2r^{1+q^3})=hc\ne 0$. Hence $Q_y(x)=y\tr(cx^2)$ is not constantly zero on $\rad(Q_y)=\F_{q^3}$ if $y\in \F_q^*$.
For each $y\not\in \F_q$, the $\F_q$-linear subspace $\{x\in \F_{q^3}: \,\tr((y^{q-q^2}+y^{q^2-q})x)=0\}$ is spanned by $y$ and $y^{-1}$. It can not contain $\F_q$, since otherwise $y$ would lie in a degree two extension of $\F_q$. It follows that $\tr(y^{q-q^2}+y^{q^2-q})\ne 0$, so $Q_y(\rad(Q_y))\ne\{0\}$ for all $y\not\in \F_q$. This proves the sufficiency part.

\end{proof}

\begin{remark}\label{remark_3dim} Take $\delta\in\F_{q^6}\setminus\F_{q^3}$ such that $\delta^{-1}+\delta^{-q^3}\in\F_q^*$, and
take the decomposition  $\F_{q^3}=T_0\oplus \F_q$ as in \cite{kwnew}, where $T_0=\{x\in\F_{q^3}:\tr_{\F_{q^3}/\F_q}(x)=0\}$. Let $W$ be the image of $\F_{q^3}$ under the map $x\mapsto \tr_{\F_{q^3}/\F_q}(x)+\delta x$. Then  $W=T_0\cdot\delta\oplus\F_q\cdot(1+\delta)$, and $\{\beta^i\cdot W:\,0\le i\le q^3\}$ forms a spread $\cS$ of type $\cC$, where $\beta$ is an element of order $(q^3+1)(q-1)$. Moreover, the spread $\cS$ is symplectic with respect to the nondegenerate alternating form
$A(x,y)=\tr_{q^6/q}\big((\delta+\delta^{q^3})^{-1}xy^{q^3}\big)$.
\end{remark}

\begin{lemma}\label{thm_even_3dim}
Let $L(X)\in\F_{q^3}[X]$ be a monic reduced $q$-polynomial and $\delta\in\F_{q^6}\setminus\F_{q^3}$. The map $x\mapsto Q(x)=(L(x)+\delta x)^{1+q^3}$ is a permutation of $\F_{q^3}$ if and only if $L(x)+\delta x=u^{-1}\tr_{q^3/\F_q}(u^qx)+\delta' x$ for some $\delta'\not\in\F_{q^3}$ and $u\in\F_{q^3}^*$ such that $\delta'^{-1}+\delta'^{-q^3}\in u^{1-q}\cdot\F_q^*$.
\end{lemma}
\begin{proof} If $L(x)+\delta x=u^{-1}\tr_{q^3/\F_q}(u^qx)+\delta' x$, then with $y=u^qx$ we have $L(x)+\delta x=u^{-1}\left(\tr_{q^3/\F_q}(y)+\delta'u^{1-q} y\right)$, and the sufficient part of the theorem follows from Lemma \ref{lemma_even_3dim}.  Therefore, we only need to prove the necessary part.

Assume that $x\mapsto Q(x)$ is a a permutation, and we need to prove that $L(x)+\delta x$ is of the desired form. If $L(x)=x^{q^2}+ax^q+bx$, then by replacing $\delta$ with $\delta+b$ we assume that $b=0$.  If $a$ is a nonzero $(q-1)$-st power, then  $a=u^{q^2-1}$ for some $u\in\F_{q^3}^*$ and $L(x)+\delta x=u^{-1}\tr_{\F_{q^3}/\F_q}(u^qx)+(u^{q-1}+\delta)x$. Similar to the first paragraph of this proof, we can deduce that $L(x)+\delta x$ is of the desired form by Lemma \ref{lemma_even_3dim}. There are two remaining cases.
\begin{enumerate}
\item In the case $a$ is not a nonzero $(q-1)$-st power, $y\mapsto u=y^{q^2}+a^qy^q$ is a permutation of $\F_{q^3}$ and $L(u)+u\delta=(1+a^q\delta)^{1+q^3}(y^{q^2}+\delta'y^q)^{1+q^3}$, where $\delta'=\frac{\delta+a^{1+q^2}}{\delta a^q+1}$. Therefore, $x\mapsto Q(x)$ is a permutation if and only if the map $x\mapsto (x^{q}+\delta'x)^{1+q^3}$ is.
\item In the case $a=0$,  $x^{q^2}+\delta x=\delta\cdot(y^q+\delta^{-1}y)$ for $y=x^{q^2}$.  Therefore, $x\mapsto Q(x)$ is a permutation if and only if the map $x\mapsto (x^{q}+\delta^{-1}x)^{1+q^3}$ is.
\end{enumerate}
We now show that $x\mapsto Q(x)=(x^q+\delta x)^{1+q^3}$ is not a permutation of $\F_{q^6}$ for any $\delta\not\in\F_{q^3}$, which will exclude these two cases and conclude the proof. In this case, we have $Q(X)=(X^q+\delta X)(X^q+\delta^{q^3} X)$. By Hermite's criterion for permutation polynomials (cf. \cite[Theorem 7.4]{ff}), $Q(X)^{q^2-1}\pmod{X^{q^3}-X}$ has degree at most $q^3-2$. The polynomial $Q(X)^{q^2-1}$ has degree at most $2q(q^2-1)<2(q^3-1)$, so its coefficient of $X^{q^3-1}$ should be zero. Since $q^2-1=(q-1)q+(q-1)$, we can rewrite $Q(X)^{q^2-1}$ as
\[
X^{2q^2-2}\cdot
\left( X^{2(q^2-q)}+s^qX^{q^2-q}+t^q\right)^{q-1}\cdot
\left( X^{2(q-1)}+sX^{q-1}+t\right)^{q-1},
\]
where $s=\delta+\delta^{q^3}$ and $t=\delta^{1+q^3}$.
The second term in the product contributes monomials that are powers of $X^{q(q-1)}$ and the third term contributes monomials that are powers of $X^{q-1}$. If $q^3-1=2(q^2-1)+q(q-1)i+(q-1)j$ with $0\le i,j\le 2(q-1)$, then we necessarily have $j=q-1$ and $i=q-2$. Thus the coefficient of $X^{q^3-1}$ in $Q(X)^{q^2-1}$ is the product of the coefficient of $X^{q-1}$ in $(X^{2 }+sX +t)^{q-1}$  and that of $X^{ q-2 }$ in $(X^{2 }+s^qX +t^q)^{q-1}$, i.e.,
\begin{equation}\label{eqn_coef}
\sum_{i=0}^{q/2}\binom{q-1}{i,q-1-2i,i}s^{q-1-2i}t^i\cdot \sum_{j=0}^{q/2-1}\binom{q-1}{j,q-2-2j,j+1}s^{q(q-2-2j)}t^{q(j+1)}.
\end{equation}
Here, the numbers $\binom{q-1}{i,j,k}$'s are  trinomial coefficients. A straightforward analysis using Lucas' theorem shows that: $\binom{q-1}{i}$ is odd for all $0\le i\le q-2$; $\binom{2i}{i}$ is odd if and only if $i=0$; $\binom{2j+1}{j}$ is odd if and only if $j=2^\ell-1$ for some nonnegative integer $\ell$. Since $\binom{q-1}{i,q-1-2i,i}=\binom{2i}{i}\binom{q-1}{2i}$ and $\binom{q-1}{j,q-2-2j,j+1}=\binom{2j+1}{j}\binom{q-1}{2j+1}$, the quantity in \eqref{eqn_coef} is equal to
\[
s^{q-1}\cdot\sum_{\ell=0}^{e-1}s^{q(q-2^{\ell+1})}t^{2^\ell q}
=s^{q^2+q-1}\cdot\sum_{\ell=0}^{e-1} (ts^{-2})^{2^\ell q},
\]
where $e$ is such that $q=2^e$. Recall that this quantity is zero. Since $s\ne 0$, we  have $\sum_{\ell=0}^{e-1}(ts^{-2})^{2^\ell }=0$ and thus $\tr_{\F_{q^3}/\F_2} (ts^{-2})=0$. By \cite[Theorem 2.25]{ff}, there exists $u\in\F_{q^3}$ such that $ts^{-2}=u^2+u$. Set $z=\delta^{q^3-1}$. We have $z\not\in\F_{q^3}$ and $ts^{-2}=(z+z^{-1})^{-1}$ by direct check. The minimal polynomial of $z$ over $\F_{q^3}$ is $(X-z)(X-z^{-1})=X^2+(u^2+u)^{-1}X+1$, but the latter has $\frac{u}{u+1}\in\F_{q^3}$ as a root: a contradiction. This completes the proof.
\end{proof}
As an immediate corollary, we get the following characterization result.
\begin{thm}
Let $q$ be even. The type $\cC$ spreads with ambient space $(\F_{q^6},+)$ and kernel $\F_q$ are isomorphic to those described in Remark \ref{remark_3dim}.
\end{thm}

We end this section with some remarks on the higher dimensional case. Theorem \ref{thm_even_4dim} supports the conjecture that there is no  $\cC$-plane of even order and even dimension.  The simple nature of the proof in the case $n=4$ suggests that the method may be applicable to larger values of even $n$, and of course new ingredients are needed to prove the conjecture. In the case where $n$ is odd, there are the constructions in \cite{kwnew}, making a characterization in this case a challenging problem. It may be more practical to classify the $\cC$-planes of order $2^p$ with $p$ an odd prime, where no such non-Desargeusian planes are known. The approach in this section provides the first step towards a complete characterization of $\cC$-planes of even order.\\

\noindent{\bf Acknowledgement.} This research was supported by the National Natural Science Foundation of China under Grant 11422112 and Fundamental Research Fund for the Central Universities of China. The author thanks the referees for detailed comments and suggestions that helped to improve the presentation of the paper. The author is indebted to Professor William M. Kantor for numerous helpful comments and suggestions during this project, and he also thanks Professor Gary L. Ebert who brought the problem to his attention in 2011.


\begin{thebibliography}{99}

\bibitem{albert_tw}A.A. Albert, On nonassociative division algebra, Trans. Amer. Math. Soc. 72 (1952), 296-309.

\bibitem{albert_iso1}A.A. Albert, Generalized twisted fields, Pacific J. Math. 11(1961), 1-8.
\bibitem{albert_iso2}A.A. Albert, Isotopy for generalized twisted fields, An. Acad. Brasil. Ci. 33 (1961), 265-275.

\bibitem{andre_tr}J. Andre, Uber nicht-Desarguesche Ebenen mit transitiver Translationsgruppe, Math. Z. 60 (1954), 156-186.

\bibitem{be_last}R.D. Baker, C. Culbert, G.L. Ebert, K.E. Mellinger, Odd order flag-transitive affine planes of dimension three over their kernel, Adv. Geometry (2003), 215-223.

\bibitem{be_baer2} R.D. Baker, J.M. Dover, G.L. Ebert, K.L. Wantz, Perfect Baer subplane partitions and three-dimensional flag-transitive planes, Des. Codes Cryptogr. 21 (2000), 19-39.

\bibitem{be_baer}R.D. Baker, J.M. Dover, G.L. Ebert, K. L. Wantz, Baer subgeometry partitions, J. Geom. 67 (2000), 23-34.

\bibitem{be_const}R.D. Baker, G.L. Ebert, Constructions of two-dimensional flag-transitive planes, Geom. Dedicata 27 (1988), 9-14.

\bibitem{be_2dim}R.D. Baker, G.L. Ebert, Two-dimensional flag-transitive planes revisited, Geom. Dedicata 63 (1996), 1-15.

\bibitem{be_x}R.D. Baker, G.L. Ebert, K.H. Leung, Q. Xiang, A trace conjecture and flag-transitive affine planes, J. Combin. Theory Ser. A 95 (2001), 158-168.

\bibitem{rtcs_sur}S. Ball, M. Lavrauw, Commutative semifields of rank $2$ over their middle nucleus, in: G.L. Mullen et al. (eds), {\it Finite Fields with Applications to Coding Theory, Cryptography and Related Areas}, Springer-Verlag, Berlin-Heidelberg, 2002.

\bibitem{helle_sur}L. Budaghyan, T. Helleseth, Planar functions and commutative semifields, Tatra Mt. Math. Publ. 45 (2010), 15-25.

\bibitem{ls}F. Buekenhout, A. Delandtsheer, J. Doyen, P. Kleidman, M. Liebeck, J. Saxl, Linear spaces with flag-transitive automorphism groups, Geom. Dedicata 36 (1990), 89-94.

\bibitem{rtcs_char1}A. Blokhuis, M. Lavrauw, S. Ball, On the classification of semifield flocks, Adv. Math. 180 (2003), 104-111.

\bibitem{CG}S.D. Cohen, M.J. Ganley, Commutative semifields, two dimensional over their middle nucleus, J. Algebra 75 (1982), 373-385.


\bibitem{coulter}R.S. Coulter, M. Henderson, Commutative presemifields and semifields. Adv. Math. 217 (2008), 282-304.

\bibitem{ch_planar}R.S. Coulter, M. Henderson, On a conjecture on planar polynomials of the form $X(Tr_n(X)-uX)$, Finite Fields Appl.  21  (2013), 30-34.

\bibitem{chk}R.S. Coulter, M. Henderson, P. Kosick, Planar polynomials for commutative semifields with specified nuclei, Des. Codes Cryptogr. 44 (2007), 275-286.

\bibitem{nearf9}P. Dembowski, {\it Finite geometries}, Springer, Berlin-Heidelberger-New York, 1968.

\bibitem{dickson_2dim}L.E. Dickson, On commutative linear algebras in which division is always uniquely possible, Trans. Amer. math. Soc. 7 (1906), 514-522.

\bibitem{dickson_r2}L.E. Dickson, Linear algebra with associativity not assumed, Duke Math. J. 1 (1935), 113-125.


\bibitem{ebert_sur}G.L. Ebert, Partitioning problems and flag-transitive planes, Rend. Circ. Mat. Palermo Ser. II Suppl. 53 (1998), 27-44.

\bibitem{hering}C. Hering, Eine nicht-desarguesche zweifach transitive affine Ebene der Ordnung $27$, Abh. Math. Sem. Univ. Hamburg 34 (1969), 203-208.

\bibitem{foulser}D.A. Foulser, Solvable flag transitive affine groups, Math. Z. 86 (1964), 191-204.

\bibitem{foulser0}D.A. Foulser, The flag-transitive collineation groups of the finite Desarguesian affine planes, Canad. J. Math. 16 (1964), 443-472.

\bibitem{handbook}N.L. Johnson, V. Jha, M. Biliotti, {\it Handbook of Finite Translation Planes}, Pure Appl. Math. (Boca Raton), vol. 289, Chapman $\&$ Hall/CRC, Boca Raton, FL, 2007.


\bibitem{kantor_even}W.M. Kantor, Spreads, translation planes and Kerdock sets, I, SIAM J. Alg. Disc. Methods 3 (1982), 151-165.

\bibitem{kantor_proj}W. Kantor, Primitive permutation groups of odd degree, and an application to finite projective planes, J. Algebra 106 (1987), 15-45.


\bibitem{kantor_odd} W.M. Kantor, Two families of flag-transitive affine planes, Geom. Dedicata 41 (1992), 191-200.

\bibitem{kantor_h}W.M. Kantor, $2$-transitive and flag-transitive designs, in: D. Jungnickel, S.A. Vanstone (Eds), {\it Coding Theory, Design Theory, Group Theory},  Wiley, New York, 1993, pp. 13-30.


\bibitem{ks}W. M. Kantor, C. Suetake, A note on some flag-transitive affine planes, J. Combin. Theory Ser. A 65 (1994), 307-310.


\bibitem{kwnew}W.M. Kantor, M.E. Williams, New Flag-Transitive Affine Planes of Even Order, J. Combin. Theory Ser. A 74 (1996), 1-13.


\bibitem{rtcs_char2}M. Lavrauw, Sublines of prime order contained in the set of internal points of  a conic, Des. Codes Cryptogr. 38 (2006), 113-123.

\bibitem{fs_cur}M. Lavrauw, O. Polverino, Finite Semifields, in: L. Storme, J.D. Beule (Eds), {\it Current Research Topics in Galois Geometry}, 2011, pp. 127-155.


\bibitem{ff}R. Lidl, H. Niederreiter, {\it Finite Fields}, 2nd ed., Cambridge University Press, Cambridge, 1997.

\bibitem{luneburg}H. Luneburg, $\ddot{U}$ber projektive Ebenen in denen jede Fahne von einer nicht-trivialen Elation invariant gelassen wird, Abh. Math. Sem. Univ. Hamburg  29 (1965), 37-76.

\bibitem{meni1}G. Menichetti, On a Kaplansky conjecture concerning  three dimensional division algebras over a finite field, J. Algebra 47 (1977), 400-410.

\bibitem{meni2}G. Menichetti, $n$-Dimensional algebras over a field with a cyclic extension of degree $n$, Geom. Dedicata 63 (1996) 69-94.

\bibitem{naka}K. Minami, N. Nakagawa, On planar functions of elementary abelian $p$-group type, Hokkaido Math. J. 37 (2008), 531-544.

\bibitem{handbookff}G.L. Mullen, D. Panario, {\it Handbook of finite fields}, CRC Press, 2013.

\bibitem{rtcs_pw}T. Penttila, B. Williams, Ovoids of parabolic spaces, Geom. Dedicata 82 (2000), 1-19.


\bibitem{prince}A.R. Prince, Flag-transitive affine planes of order at most 125, J. Geom.  67  (2000), 208-216.

\bibitem{suetake}C. Suetake, Flag transitive planes of order $q^n$ with a long cycle $\ell_\infty$ as a collineation, Graphs Combin. 7 (1991), 183-195.

\bibitem{suetake2}C. Suetake, On flag-transitive affine planes of order $q^3$, Geom. Dedicata 51 (1994), 123-131.

\bibitem{Kthas_proj}K. Thas, Finite flag-transitive projective planes: a survey and some remarks, Discrete Math. 266 (2003), 417-429.

\bibitem{tz_proj}K. Thas, D. Zagier, Finite projective planes, Fermat curves, and Gaussian periods, J. Eur. Math. Soc. 10 (2008), 173-190.


\bibitem{wagner}A. Wagner, On finite affine line transitive planes, Math. Z. 87 (1965), 1-11.

\bibitem{weng}G. Weng, X. Zeng, Further results on planar DO functions and commutative
semifields, Des. Codes Cryptogr. 63 (2012), 413-423.

\bibitem{cs_zp}Y. Zhou, A. Pott, A new family of semifields with $2$ parameters, Adv.  Math. 234 (2011), 43-60.

\end{thebibliography}
\end{document}